\documentclass[12pt]{article}

\usepackage{amssymb}
\usepackage{graphics}

\newcommand{\qed}{$\;\;\;\Box$}
\newenvironment{proof}{\par\smallbreak{\sl Proof.~}}
{\unskip\nobreak\hfill \qed \par\medbreak}

\newcounter{claim}
\renewcommand{\theclaim}{\arabic{claim}}
\newenvironment{claim}{\refstepcounter{claim}%
\par\medskip\par\noindent{\bf Claim~\theclaim.}\rm}%
{\par\medskip\par}

\newcounter{cclaim}
\renewcommand{\thecclaim}{\arabic{cclaim}}
\newenvironment{cclaim}{\refstepcounter{cclaim}%
\par\medskip\par\noindent{\bf Claim~\thecclaim.}\rm}%
{\par\medskip\par}

\newenvironment{subproof}{\par\noindent{\it Proof of Claim.}}%
{$\Box$\par\smallbreak}
\newcommand{\hide}[1]{}

\newcommand{\N}{{\mathbb N}}

\newcommand{\R}{{\mathbb R}}
\newcommand{\C}{{\mathbb C}}
\newcommand{\Z}{{\mathbb Z}}

\newcommand{\B}{{\cal B}}
\newcommand{\CC}{{\cal C}}
\newcommand{\G}{{\cal G}}
\newcommand{\F}{{\cal F}}

\newcommand{\LL}{{\cal L}}

\newcommand{\beq}{\begin{equation}}
\newcommand{\ee}{\end{equation}}

\renewcommand{\d}{\partial}

\newtheorem{thm}{Theorem}[section]
\newtheorem{lemma}[thm]{Lemma}

\newtheorem{defn}[thm]{Definition}

\newtheorem{rem}[thm]{Remark}

\newcommand{\al}{\alpha}
\newcommand{\be}{\beta}
\newcommand{\ga}{\gamma}
\newcommand{\de}{\delta}
\newcommand{\eps}{\varepsilon}
\newcommand{\vphi}{\varphi}
\newcommand{\la}{\lambda}
\newcommand{\om}{\omega}

\newcommand{\reff}[1]{(\ref{#1})}

\newcommand{\im}{\mathop{\rm im}}

\renewcommand{\Im}{\mathop{\mathrm{Im}}\nolimits}
\renewcommand{\Re}{\mathop{\mathrm{Re}}\nolimits}

\title{Periodic Solutions to Dissipative Hyperbolic Systems. II:
Hopf Bifurcation for Semilinear Problems
} 

\newcounter{thesame}
\author{
I.~Kmit
\ \ \ L.~Recke\\
{\small
Institute of Mathematics, Humboldt University of Berlin,}
\\
{\small Rudower Chaussee 25,
\small D-12489 Berlin, Germany and
}
\\
{\small
Institute for Applied Problems of Mechanics and Mathematics,
}
\\
{\small
Ukrainian Academy of Sciences,
Naukova St.\ 3b,
79060 Lviv, Ukraine
}
\\
{\small   E-mail:
{\tt kmit@informatik.hu-berlin.de}}\\[5mm]
{\small
Institute of Mathematics, Humboldt University of Berlin,}\\
{\small 
Rudower Chaussee 25,
D-12489 Berlin, Germany}\\
{\small   E-mail:
{\tt recke@mathematik.hu-berlin.de}}
}

\date{}

\begin{document}

\maketitle

\begin{abstract}
\noindent 
We consider boundary value problems for semilinear hyperbolic systems of the type
$$
\partial_tu_j  + a_j(x,\la)\partial_xu_j + b_j(x,\la,u)  = 0, \; x\in(0,1), \;j=1,\dots,n
$$
with smooth coefficient functions $a_j$ 
and $b_j$
such that 
$b_j(x,\la,0) = 0$ for all $x \in [0,1]$, $\la \in \R$, and $j=1,\ldots,n$.
We state conditions for Hopf bifurcation, i.e.,
for existence, local uniqueness (up to phase shifts), smoothness and smooth dependence
on $\la$
of time-periodic solutions bifurcating from the zero stationary solution. Furthermore,
we derive a formula which determines the bifurcation direction.

The proof is done by means of a Liapunov-Schmidt reduction procedure. 
For this purpose, Fredholm properties of the linearized 
system and implicit function
theorem techniques are used. 

There are at least two distinguishing features of Hopf bifurcation theorems for hyperbolic PDEs in comparison with those for parabolic PDEs or for ODEs:
First, the question if a non-degenerate time-periodic solution depends smoothly on the system parameters 
is much more delicate. And second, 
a sufficient amount of dissipativity is needed in the system, and  a priori
it is not clear how to verify this in terms of the data of the PDEs and of the boundary conditions. 
\end{abstract}



\section{Introduction}\label{sec:intr}
\renewcommand{\theequation}{{\thesection}.\arabic{equation}}
\setcounter{equation}{0}

\subsection{Problem and Main Result}\label{sec:results}

This paper concerns  Hopf bifurcation for hyperbolic systems of semilinear first-order PDEs in one space dimension of the type 
\beq\label{eq:1.1}
\om\partial_tu_j  + a_j(x,\la)\partial_xu_j + b_j(x,\la,u)  = 0,\quad x\in(0,1), \;j=1,\dots,n
\ee
with reflection boundary conditions
\beq\label{eq:1.2}
\begin{array}{l}
\displaystyle
u_j(0,t) = \sum\limits_{k=m+1}^nr_{jk}u_k(0,t),\quad  j=1,\ldots,m,\\
\displaystyle
u_j(1,t) = \sum\limits_{k=1}^mr_{jk}u_k(1,t), \quad  j=m+1,\ldots,n
\end{array}
\ee
and time-periodicity conditions
\beq\label{eq:1.3}
u_j(x,t+2\pi) = u_j(x,t),\quad x\in[0,1],\;j=1,\dots,n.
\ee
Our main result (Theorem \ref{thm:hopf} below) is quite similar to Hopf bifurcation theorems 
for parabolic PDEs or for ODEs. But 
there are several distinguishing features of the proofs of Hopf bifurcation theorems 
for hyperbolic PDEs in comparison with those for parabolic PDEs or for ODEs.
First, the question of Fredholm solvability of the linearized problem
(in appropriate spaces of time-periodic functions) is essentially more difficult.
Second, the question if a non-degenerate time-periodic solution of the nonlinear problem 
depends smoothly on the system parameters 
is much more delicate. And third, 
a sufficient amount of dissipativity is needed in order to prevent  small
divisors from coming up, and we present an explicit sufficient condition \reff{Fred}
for that in terms of the data of the PDEs and of the boundary conditions. 

For problem  \reff{eq:1.1}--\reff{eq:1.3}, suppose that 
$m< n$ are positive integers
and $r_{jk} \in \R$. Moreover, 
\beq\label{smooth}
a_j : [0,1]\times\R\to\R \mbox{ and }  b_j : [0,1]\times\R\times\R^n\to\R \mbox{ are } C^\infty\mbox{-smooth}
 \mbox{ for all }   j=1,\dots,n,
\ee
\beq \label{hyp0}
a_j(x,0)\ne 0  \mbox{ for all }  x\in[0,1] \mbox{ and } j=1,\dots,n,
\ee 
\beq \label{hyp}
a_j(x,0)\ne a_k(x,0) \mbox{ for all } x\in[0,1]  \mbox{ and } 1\le j\ne k\le n.
\ee
The number $\om>0$ and the function $u=(u_1,\dots,u_n): [0,1]\times\R\to\R^n$ are the state parameters to be determined,
and  $\la\in\R$ is the bifurcation parameter. Speaking about solutions to  (\ref{eq:1.1})--(\ref{eq:1.3}), throughout 
the paper we mean classical solutions, i.e., $C^1$-smooth maps $u: [0,1]\times\R\to\R^n$ which satisfy  (\ref{eq:1.1})--(\ref{eq:1.3})
pointwise.

If, for given $\la$ and $\om$, $u$ is a solution to (\ref{eq:1.1})--(\ref{eq:1.3}) and  $U : [0,1]\times\R\to\R^n$ 
is defined as $U(x,t):= u(x,\om t)$, then
\beq\label{U}
\partial_tU_j  + a_j(x,\la)\partial_xU_j + b_j(x,\la,U)  = 0,\quad x\in(0,1),\;j=1,\dots,n,
\ee
and $U$ is $2\pi/\om$-periodic with respect to time; and vice versa. In other words: Solutions to (\ref{eq:1.1})--(\ref{eq:1.3})
correspond to  $2\pi/\om$-periodic solutions to (\ref{U}) (which satisfy the boundary conditions (\ref{eq:1.2})).

We suppose that for all  $\la$ and $\om$  the function $u=0$ is a solution 
(the so-called trivial solution)
to (\ref{eq:1.1})--(\ref{eq:1.3}), i.e., 
\beq\label{u=0}
b_j(x,\la,0)=0  \mbox{ for all } x\in[0,1], \la\in\R, \mbox{ and } j=1,\dots,n.
\ee
We are going to describe families of non-stationary solutions to (\ref{eq:1.1})--(\ref{eq:1.3}) bifurcating from
the family
of trivial solutions.
With this aim  we consider the following eigenvalue problem for the linearization of (\ref{eq:1.1})--(\ref{eq:1.3})
at the trivial solution
\beq\label{evp}
\begin{array}{rl}
\displaystyle
a_j(x,\la)\frac{d}{dx}v_j + \sum_{k=1}^n\d_{u_k}b_j(x,\la,0)v_k  = \mu v_j, &x\in(0,1),\;j=1,\dots,n\\
\displaystyle
v_j(0) = \sum\limits_{k=m+1}^nr_{jk}v_k(0),& j=1,\ldots,m,\\
\displaystyle
v_j(1) = \sum\limits_{k=1}^mr_{jk}v_k(1), & j=m+1,\ldots,n
\end{array}
\ee
and the corresponding adjoint eigenvalue problem 
\beq\label{adevp}
-\frac{d}{dx}\left(a_j(x,\la)w_j\right) + \sum_{k=1}^n\d_{u_j}b_k(x,\la,0)w_k  = \nu w_j,\quad x\in(0,1),\;j=1,\dots,n,\\
\ee
\beq
\label{adevpbc}
\begin{array}{lll}
\displaystyle
a_j(0,\la)w_j(0) = -\sum\limits_{k=1}^mr_{kj}a_k(0,\la)w_k(0),\; j=m+1,\ldots,n,\\
\displaystyle
\displaystyle
a_j(1,\la)w_j(1) = -\sum\limits_{k=m+1}^nr_{kj}a_k(1,\la)w_k(1), \;  j=1,\ldots,m.
\end{array}
\ee
Here $\mu,\nu\in\C$ are the eigenvalues and $v=(v_1,\dots,v_n), w=(w_1,\dots,w_n) : [0,1]\to\C^n$
are the  corresponding eigenfunctions.

Let us formulate our assumptions which are analogous to corresponding assumptions in Hopf bifurcation for ODEs
(see, e.g., \cite{Andronov,ChowHale,Hopf}) and for parabolic PDEs (see, e.g., \cite{CrandRab,Haragus,Henry,Ki}).

The first assumption states that there exists a pure imaginary pair of geometrically simple eigenvalues to 
(\ref{evp}) with $\la=0$:
\beq \label{geo}
\left.
\begin{array}{l}
\mbox{For } \la=0 \mbox{ and } \mu=i \mbox{ there exists exactly one (up
to linear dependence)}\\
\mbox{solution } v\ne 0 \mbox{ to } (\ref{evp}).
\end{array}
\right\}
\ee

The second assumption states that the eigenvalues $\mu=\pm i$ to (\ref{evp}) with
$\la=0$ are algebraically simple:
\beq \label{alg}
\left.
\begin{array}{l}
\mbox{For any solution } v\ne 0 \mbox{ to } (\ref{evp})  \mbox{ with } \la=0 \mbox{ and } \mu=i
\mbox{ and }\\ \mbox{for any solution  }
\displaystyle w\ne 0 \mbox{ to }  (\ref{adevp})-\reff{adevpbc} \mbox{ with }  \la=0 \mbox{ and }
\nu=-i \\\mbox{we have  }
\displaystyle\sum\limits_{j=1}^n\int_0^1v_j(x)\overline{w_j(x)}\,dx\ne 0.
\end{array}
\right\}
\ee
In what follows, we fix a solution $v=v^0$ to (\ref{evp}) with $\la=0$ and $\mu=i$ and a solution  
$w=w^0$ to (\ref{adevp})--\reff{adevpbc} with  $\la=0$ and $\mu=-i$ such that
\beq \label{normed}
\sum\limits_{j=1}^n\int_0^1v_j^0(x)\overline{w_j^0(x)}\,dx=2.
\ee

The third assumption is the so-called transversality condition. It states that the eigenvalue $\mu=\mu(\la) \approx i$
to \reff{evp} with $\la \approx 0$ crosses the imaginary axis transversally
if $\la$ crosses zero:
\beq \label{eq:IV}
\al:=\frac{1}{2}\Re\sum\limits_{j=1}^n\int_0^1\left(\d_\la a_j(x,0)\frac{d}{dx}v_j^0(x)
+\sum\limits_{k=1}^n\d_\la \d_{u_k}b_{j}(x,0,0)v_k^0(x)\right)
\overline{w_j^0(x)}\,dx\ne 0
\ee
(in fact it holds $\Re \mu'(0)=\al$, cf. \reff{trans}).

The fourth assumption is the so-called nonresonance condition:
\beq \label{nonres}
\mbox{If }  (\mu,v) \mbox{ is a solution to } (\ref{evp})  \mbox{ with } \la=0, \, \mu=ik, \, k \in \Z, 
\mbox{ and } v \not= 0, \mbox{ then } k=\pm 1.
\ee
In order to formulate our main result we need some more notation: 
\beq
\label{bb}
b_{jk}(x) := \partial_{u_k}b_j(x,0,0), \; b_{jkl}(x) := \partial^2_{u_ku_l}b_j(x,0,0), \;
b_{jklr}(x) := \partial^3_{u_k u_l u_r}b_j(x,0,0),
\ee
\begin{eqnarray}
\label{Rdef0}
&&R_0:=\max_{1\le j \le m} \;\max_{0 \le x \le 1} 
\sum_{k=m+1}^n|r_{jk}| \;\exp \left(-\int_0^x\frac{b_{jj}(\xi)}{a_j(\xi,0)} d\xi\right),\\
\label{Rdef1}
&&R_1:=\max_{m+1 \le j \le n} \;\max_{0 \le x \le 1} 
\sum_{k=1}^m|r_{jk}|\;\exp \int_x^1\frac{b_{jj}(\xi)}{a_j(\xi,0)} d\xi,
\end{eqnarray}
and 
\beq\label{coef}
\beta:=
-\mbox{Re} \int_0^1 \left(\frac{1}{8}\sum_{j,k,l,r=1}^n b_{jklr}v_k^0 v_l^0 \overline{v_r^0}\, \overline{w_j^0}
+\sum_{j,k,l=1}^n b_{jkl} \left(y_k\overline{v_l^0} + z_kv_l^0\right) \overline{w_j^0}\right)dx.
\ee
Here $y:[0,1]\rightarrow \mathbb{C}^n$ and  $z:[0,1]\rightarrow \mathbb{R}^n$ are the solutions to the boundary value problems
$$
\begin{array}{rcll}
\displaystyle
a_j(x,0)\frac{d}{dx} y_j - 2iy_j + \sum_{k=1}^n b_{jk}(x)y_k&=& 
\displaystyle-\frac{1}{4}\sum_{k,l=1}^n b_{jkl}(x)v_k^0(x)v_l^0(x),
&j=1,\dots,n \\
\displaystyle
 y_j(0) &=& 
\displaystyle\sum_{k=m+1}^n r_{jk} y_k(0), & j=1,\dots,m, \\
\displaystyle
y_j(1) &=& 
\displaystyle\sum_{k=1}^m r_{jk}y_k(1), & j=m+1,\dots,n
\end{array}
$$
and
$$
\begin{array}{rcll}
\displaystyle
a_j(x,0)\frac{d}{dx} z_j +\sum_{k=1}^n b_{jk}(x)z_k
&=& 
\displaystyle  -\frac{1}{2}\sum_{k,l=1}^n b_{jkl}(x)v_k^0(x) \overline{v_l^0(x)}, &j=1,\dots,n \\
\displaystyle
 z_j(0) &=& \displaystyle \sum_{k=m+1}^n r_{jk} z_k(0), & j=1,\dots,m, \\
\displaystyle
z_j(1) &=& \displaystyle\sum_{k=1}^m r_{jk} z_k(1), & j=m+1,\dots,n.
\end{array}
$$
Remark that the two boundary value problems above are uniquely solvable due to assumption~\reff{nonres}.

\begin{defn}\label{Cspaces}
We denote by $\CC_n$ the space of all continuous maps $u:[0,1] \times \R \to \R^n$, which satisfy the 
time-periodicity condition \reff{eq:1.3}, with the norm 
$$
\|u\|_\infty:=\max_{1 \le j \le n} \max_{0 \le x \le 1} \max_{t \in \R}|u_j(x,t)|.
$$
\end{defn}

Now we are prepared to formulate our main result:
\begin{thm}
\label{thm:hopf}
Suppose \reff{smooth}--\reff{hyp}, \reff{u=0},  \reff{geo}--\reff{nonres}, and
\beq
\label{Fred}
R_0 R_1 <1.
\ee 
Then there exist $\eps_0>0$ and $C^2$-maps
$\hat\la : [0,\eps_0]\to\R$, $\hat\om : [0,\eps_0]\to\R$, and
$\hat u : [0,\eps_0]\to \CC_n$  such that the following is true:

(i) {\rm Existence of nontrivial solutions:} For all $\eps\in (0,\eps_0]$ the function $\hat u(\eps)$
is a $C^\infty$-smooth nontrivial solution  to (\ref{eq:1.1})--(\ref{eq:1.3})
with  $\la=\hat \la(\eps)$ and $\om=\hat \om(\eps)$.

(ii)  {\rm Asymptotic expansion and bifurcation direction:} We have 
$$
\hat\la(0)= \hat\la'(0)=0,\;  \hat\la''(0)=\frac{\be}{\al},\;  \hat\om(0)=1,
$$
and
$
\hat u(0)(x,t)=0, \;  \hat u'(0)(x,t)=\Re v^0(x)\cos t - \Im v^0(x)\sin t \mbox{ for all } x \in [0,1] \mbox{ and } t \in \R.
$

(iii) {\rm Local uniqueness:} There exists $\de>0$ such that for all 
nontrivial solutions  to (\ref{eq:1.1})--(\ref{eq:1.3}) with 
$ |\la|+|\om-1| + \|u\|_{\infty}<\de$
there exist $\eps\in (0,\eps_0]$ and $\vphi\in\R$ such that 
$\la=\hat\la(\eps)$,  $\om=\hat\om(\eps)$, and $u(x,t)=\hat u(\eps)(x,t+\vphi)$ for all $x\in[0,1]$ and $t\in\R$.
\end{thm}

Our paper is organized as follows:

In Section  \ref{Applications} we comment about mathematical models which contain 
dissipative hyperbolic PDEs and which are used for describing destabilization of stationary states and/or for describing stable time-periodic processes.

In Section  \ref{related} we comment about some publications which are related to ours.

In Section \ref{Function Spaces and Operators} we derive a weak formulation \reff{rep1}, \reff{rep2}  for the PDE problem \reff{eq:1.1}--\reff{eq:1.3}
via integration along characteristics, and we introduce operators in order to 
write this weak formulation as the operator equation \reff{abstract}.

In Section \ref{Lyapunov-Schmidt Procedure} 
we do a Liapunov-Schmidt procedure (as it is known for Hopf bifurcation for parabolic PDEs or for ODEs)  in order to reduce 
(for $\la \approx 0, \om \approx 1$ and $u \approx 0$)
the problem  \reff{eq:1.1}--\reff{eq:1.3} with infinite-dimensional state parameter $(\om,u) \in \R \times \CC_n$ to a problem with two-dimensional state parameter.
Here the main technical results are  Lemmas \ref{Magnus} and \ref{infIFT} about local unique solvability of the infinite dimensional part 
of the Liapunov-Schmidt system and
Lemma \ref{laregular} about smooth dependence of the solution on parameters.
The proofs of these Lemmas are much more complicated than the corresponding proofs for parabolic PDEs or for ODEs.
The point is that  in the case of dissipative hyperbolic 
PDEs the question of Fredholm solvability of linear periodic problems 
as well as the question of
smooth dependence of solutions on parameters are  much more difficult.
The difficulty with the Fredholmness is solved in \cite{KR3}.

In Section \ref{bifeq} we put the solution of the  infinite dimensional part
of the Liapunov-Schmidt system into the  finite dimensional part and discuss the behavior of the resulting
equation. This is completely analogous to what is known from Hopf bifurcation for ODEs and parabolic PDEs.

In Section \ref{Examples} we   present an example of a problem of the type  (\ref{eq:1.1})--(\ref{eq:1.3}) such that all assumptions 
\reff{smooth}--\reff{hyp}, \reff{u=0},  \reff{geo}--\reff{nonres}, and \reff{Fred} of Theorem \ref{thm:hopf} are satisfied.

Finally, in the appendix  we present a simple linear version of the so-called fiber contraction principle, which is used in the proof of the key technical Lemma \ref{laregular}.

\subsection{Generalizations}

We do not know if generalizations of Theorem \ref{thm:hopf} to higher space dimensions and/or to
quasilinear systems exist and how they should look like. On the other hand,  Theorem \ref{thm:hopf} 
can be generalized to  equations and boundary conditions with nonlocal terms
(periodic boundary conditions, for example). Then the dissipativity condition \reff{Fred} should be changed accordingly,
which is a task of a different paper.

Also, it is plausible that  Theorem \ref{thm:hopf} generalizes to  second-order semilinear
wave equations with Dirichlet, Neumann, Robin, or periodic boundary conditions, and, again, a nontrivial
question  is how we have to modify the condition \reff{Fred}.
Remark that in \cite{Koch} a Hopf bifurcation theorem is stated without proof for
 second-order quasilinear  hyperbolic systems with arbitrary space dimension
subjected to homogeneous Dirichlet boundary conditions.

\subsection{Applications}
\label{Applications}

One field of appearance of models, which contain  dissipative hyperbolic PDEs and which are used 
for describing, optimizing and stabilizing time-periodic processes,
is modeling of semiconductor laser devices and  their applications in communication systems (see, e.g. \cite{LiRadRe,Peterhof,Rad,RadWu,Sieber}).
Remark that the 
mentioned semiconductor laser models have some specific features: There the hyperbolic PDEs (balance equations for the complex amplitudes of the light field) have complex coefficients,
and they are coupled with ODEs (balance equations for the electron densities). Moreover, the models possess a nonlinear-Schr\"odinger-equation-like $SO(2)$-equivariance,
therefore the Hopf bifurcations are bifurcations from relative equilibria (rotating waves) into relative periodic orbits (modulated waves). Anyway, for proving Hopf bifurcation there
one has to overcome the same problems as in the present paper.

In \cite{Barbera} (with applications to population dynamics), \cite{Hillen} (with applications to correlated random walks), \cite{Horst} (with applications to Brownian motion) and \cite{Martel} 
(with applications to Rayleigh-B\'{e}nard convection) the authors considered semilinear hyperbolic systems of the type (\ref{eq:1.1}) with boundary conditions of the type  (\ref{eq:1.2})
and with certain additional structures (determined by the applications) in the PDEs as well as in the boundary conditions. 
A linear stability analysis for the stationary solutions is done, and the bifurcation hypersurfaces in the space of control parameters are described, in particular
those where the conditions \reff{geo}, \reff{alg}, and \reff{eq:IV} are satisfied and, hence, where the 
authors expect Hopf bifurcation to appear.

\subsection{Some remarks on related work}
\label{related}
The main methods for proving Hopf bifurcation theorems are, roughly speaking, center manifold reduction 
and Liapunov-Schmidt reduction.
In order to apply them to abstract evolution equations one needs to have a smooth center manifold for 
the corresponding 
semiflow (for center manifold reduction) or a Fredholm property  of the linearized equation on 
spaces of periodic functions (for Liapunov-Schmidt reduction). 

In  \cite{CrandRab,Ki} Hopf bifurcation theorems for abstract evolution equations are proved by means of
 Liapunov-Schmidt reduction, and in \cite{Haragus,Ma,Vander} by means of center manifold reduction.
In  \cite{CrandRab,Ki} it is assumed that the operator of the linearized equation is sectorial
(see \cite[Hypothesis (HL)]{CrandRab} and \cite[Hypothesis I.8.8]{Ki}), hence this setting is not appropriate 
for hyperbolic PDEs. In  \cite{Haragus,Ma,Vander} the assumptions concerning the linearized operator are more 
general, including non-sectorial operators. However, it is unclear if our system 
\reff{eq:1.1}--\reff{eq:1.3} can be written as an abstract evolution equation satisfying those conditions.
On the other hand, it is interesting to see in \cite{Vander} that the 1D semilinear damped wave equation   
$\d_t^2u=\d_x^2u-\gamma \d_tu+f(u)$
with  $f(0)=0$, subjected to homogeneous Dirichlet boundary conditions, can be written as an abstract 
evolution equation satisfying the general assumptions of  \cite{Vander}. But for that it is essentially used the
exeptional property that the nonlinearity $f$ maps functions with zero boundary conditions into functions which have zero
boundary conditions again.

The celebrated counter example of M. Renardy \cite{Renardy}
shows that a reasonable linear hyperbolic differential operator in two space dimensions with  periodic boundary conditions
may not satisfy the spectral mapping property 
(see also \cite{ChiconeL, Engel}), and, hence, its spectral decomposition does not create a corresponding spectral decomposition 
of the corresponding linear semiflow. 
Therefore, the question of existence of center manifolds for small nonlinear perturbations of this linear semiflow is completely open.

But hyperbolic PDEs in one space dimension are better: They satisfy, under reasonable assumptions, the spectral mapping property in $L^p$-spaces 
(see \cite{Lopes}) as well as in $C$-spaces (see \cite{Lichtner1}). The spectral mapping property in $L^p$-spaces is used in \cite{Peterhof, Sieber}
to show the existence of smooth center manifolds for linear first-order hyperbolic PDE systems which are coupled with nonlinear ODEs.
Here the linearity of the problem with respect to the infinite dimensional part of the phase space is essential, because it implies the 
well-posedness and smoothness  in $L^p$-spaces of the Nemyckii operators. 
The spectral mapping property in $C$-spaces is used in \cite{Lichtner2} to show the way  how to prove the  
existence of smooth center manifolds  in $C$-spaces for general semilinear first-order hyperbolic systems. It 
seems that 
going this way one could prove the Hopf bifurcation theorem  of the the present paper as well.

The eigenvalue problem \reff{evp} is well-understood. 
The set of the real parts of all eigenvalues is bounded from above.
All eigenvalues have finite multiplicity. The eigenvalues are asymptotically (for large imaginary parts) close to eigenvalues of the 
corresponding ``diagonal'' eigenvalue problem (i.e., if the non-diagonal terms $\d_{u_k}b_j(x,\la,0)v_k$ with $j\not=k$ in \reff{evp} are neglected).
If \reff{Fred} is satisfied, then the supremum of the real parts of the eigenvalues of the ``diagonal'' eigenvalue problem is negative, and, hence,
only finitely many eigenvalues of the ``full''  eigenvalue problem \reff{evp} can be close to the imaginary axis.
For related rigorous statements see, e.g., \cite{Lichtner1}, \cite[Chapter 6.1]{Mogul} and \cite{Neves,Lopes,ReSchne}.

In \cite{Akramov}  Hopf bifurcation for (\ref{eq:1.1})--(\ref{eq:1.3}) with $a_j(x,\la)$ not depending on $\la$ is considered. It is assumed that 
many reflection coefficients 
$r_{jk}$ vanish, which allows to use some smoothing property for the solutions to the corresponding linearized initial-boundary value problem~\cite{lyulko}.
Despite a number of interesting ideas appearing in \cite{Akramov}, it seems 
that there is an essential gap in the realization of the Liapunov-Schmidt procedure: Using our notation,  
the finite-dimensional part \reff{finite} of the   Liapunov-Schmidt system is first locally solved 
with respect to $\la$ and $\mu$ (in terms of $v$ and $w$):
$$
\la=\hat{\la}(v,w),\; \om=\hat{\om}(v,w).
$$
Then these solutions are inserted into the  infinite-dimensional part \reff{infinite} of the  Liapunov-Schmidt system. 
Finally the intention is to solve the resulting equation by means of the implicit function theorem with respect to $w$. 
But in the  resulting equation there appear terms of the type
$$
w_j(\xi,\tau_j(\xi,x,t,\hat{\om}(v,w))).
$$
In other words: The unknown function appears in the argument of the  unknown function 
(like, for example, in ODEs with state depending delay).
If one formally differentiates this expression with respect to $w$, then there appears $\d_tw_j(\xi,\tau_j(\xi,x,t,\hat{\om}(v,w)))$,
which has less smoothness with respect to $t$ than $w_j$. Roughly speaking, this loss of smoothness property is the reason why  
the nonlinear operator corresponding to the resulting equation
is not differentiable in a neighborhood of zero, neither in the sense of $C$-spaces nor of $C^1$-spaces. Hence, the implicit function theorem 
(at least the classical one) is not applicable.

In \cite{Magal} the authors  considered 
scalar linear first-order PDEs of the type
$(\partial_t  +\partial_x + \mu)u = 0$ on $(0,\infty)$
with a nonlinear integral boundary condition at $x=0$:
$$
u(0,t) = h\left(\int_0^\infty\ga(x)u(x,t)\, dx\right).
$$
The nonlocality of the boundary condition is essential for the applied techniques in \cite{Magal} 
(integrated semigroup theory, see also  \cite{Liu}) 
to get existence of 
center manifolds and Hopf bifurcation. 
It is easy to realize that we could also consider (linear or nonlinear) integral boundary conditions, i.e., if we would replace
(\ref{eq:1.2}) by boundary conditions of the type
\begin{eqnarray*}
u_j(0,t) = h_j\left(\int_0^1\gamma_j(x)u(x,t)\, dx\right),&  j=1,\ldots,m,\\
u_j(1,t) = h_j\left(\int_0^1\gamma_j(x)u(x,t)\, dx\right),&  j=m+1,\ldots,n,
\end{eqnarray*}
then we would get essentially the same result as that described in Theorem \ref{thm:hopf},
even without any assumption of the type \reff{Fred}.  Roughly speaking, the reason is that 
the weak formulation of the problem will be of the type \reff{abstract} again, but now with $C(\la,\om)$ being compact
(due to a smoothing property proved in  \cite{kmit}).

To the best of our knowledge, almost no results  exist 
concerning smooth dependence on parameters of non-degenerate time-periodic solutions of dissipative hyperbolic PDEs.
The papers \cite[Chapter 3.5.1]{Raugel}, \cite{Raugel1} show the difficulty of this problem. There conditions are formulated such that a  non-degenerate time-periodic solution
to a system of semilinear damped wave equations survives under small parameter perturbations, but nothing is known if 
the perturbed solution depends smoothly on the perturbation parameters. Results about smooth dependence on data for 
time-periodic solutions to linear first-order hyperbolic systems with reflection boundary conditions  are given
in \cite{KR2}.

\section{Abstract formulation of  \reff{eq:1.1}--\reff{eq:1.3}}
\label{Function Spaces and Operators}
\renewcommand{\theequation}{{\thesection}.\arabic{equation}}
\setcounter{equation}{0}

In this section we derive (for $\la \approx 0$) a weak formulation for the PDE problem 
\reff{eq:1.1}--\reff{eq:1.3}
via integration along characteristics, and we introduce operators in order to 
write this weak formulation as an operator equation.

Let $\delta_0>0$ be sufficiently small such that (cf. assumption \reff{hyp0})
$$
a_j(x,\la) \not= 0 \mbox{ for all } j=1,\ldots,n,\; x \in [0,1] \mbox{ and } \la \in [-\delta_0,\delta_0].
$$
Straightforward calculations (cf. \cite[Section 2]{KR3}) show that a $C^1$-map $u:[0,1]\times \R \to \R^n$ 
is a solution to the PDE problem 
\reff{eq:1.1}--\reff{eq:1.3} with $\la \in [-\delta_0,\delta_0]$ if and only if it is a solution to the following system of integral equations:
\begin{eqnarray}
\label{rep1}
\lefteqn{
u_j(x,t)=c_j(0,x,\la)\sum_{k=m+1}^nr_{jk}u_k(0,\tau_j(0,x,t,\la,\om))}\nonumber\\
&&+\int_0^x \frac{c_j(\xi,x,\la)}{a_j(\xi,\la)}f_j(\xi,\lambda,u(\xi,\tau_j(\xi,x,t,\lambda,\om)))d\xi,\quad j=1,\ldots,m,
\end{eqnarray}
\begin{eqnarray}
\label{rep2}
\lefteqn{
u_j(x,t)=c_j(1,x,\la)\sum_{k=1}^m r_{jk}u_k(1,\tau_j(1,x,t,\la,\om))}\nonumber\\
&&-\int_x^1 \frac{c_j(\xi,x,\la)}{a_j(\xi,\la)}f_j(\xi,\la,u(\xi,\tau_j(\xi,x,t,\la,\om)))d\xi,\quad j=m+1,\ldots,n.
\end{eqnarray}
Here
\beq
\label{charom}
\tau_j(\xi,x,t,\la,\om):=\om\int_x^\xi \frac{d\eta}{a_j(\eta,\la)}+t
\ee
is the $j$-th characteristic of the hyperbolic system \reff{eq:1.1} and
\begin{eqnarray}
\label{cdef}
c_j(\xi,x,\la)&:=&\exp \int_x^\xi
\frac{\partial_{u_j}b_j(\eta,\la,0)}{a_{j}(\eta,\la)}\,d\eta,\\
\label{fdef}
f_j(x,\la,u)&:=&\partial_{u_j}b_j(x,\la,0)u_j-b_j(x,\la,u).
\end{eqnarray}
If we set $b(x,\la,u):=(b_1(x,\la,u),\ldots,b_n(x,\la,u))$ and  
$
f(x,\la,u):=(f_1(x,\la,u),\ldots,f_n(x,\la,u)),
$
then the nonlinear map $-f(x,\la,\cdot)$ is the difference of the nonlinear map $b(x,\la,\cdot)$ and of  the diagonal part of  the linear map
$\d_ub(x,\la,0)$. Hence,
the diagonal part of $\d_uf(x,\la,0)$ vanishes. This will be used later on (see \reff{dF1} and the text there), 
because this implies that the linear operators $I-C(\la, \om) -\partial_uF(\la,\om,0)$ (see \reff{Cdef} and \reff{Fdef})
are Fredholm of index zero  from $\CC_n$ into $\CC_n$.

Let us show that any solution to  \reff{eq:1.1}--\reff{eq:1.3} satisfies  \reff{rep1}--\reff{rep2}:
If $u$ is a  $C^1$-solution to  \reff{eq:1.1}--\reff{eq:1.3}, then
\begin{eqnarray}
&&\frac{d}{d\xi} u_j(\xi,\tau_j(\xi,x,t,\la,\om))+\frac{\partial_{u_j}b_j(\xi,\la,0)}{a_j(\xi,\tau_j(\xi,x,t,\la,\om))}
u_j(\xi,\_j(\xi,x,t,\la,\om))\nonumber\\
&&=\frac{f_j(\xi,\la,u(\xi,\tau_j(\xi,x,t,\la,\om)))}{a_j(\xi,\tau_j(\xi,x,t,\la,\om))}.
\label{Deq}
\end{eqnarray}
Now, applying the variation of constants formula and using the boundary conditions
\reff{eq:1.3},  one gets \reff{rep1}--\reff{rep2}.

And vice versa: For any $C^1$-smooth map $u: [0,1] \times \R \to \R$ it holds

\begin{eqnarray}
\lefteqn{\left(\om\d_t+a_j(x,\la)\d_x\right)\left(c_j(\xi,x,\la)u_k(\xi,\tau_j(\xi,x,t,\la,\om))\right)}\nonumber\\
&&=-\partial_{u_j}b_j(x,\la,0) c_j(\xi,x,\la)u_k(\xi,\tau_j(\xi,x,t,\la,\om))
\label{71}
\end{eqnarray}
and
\begin{eqnarray}
\lefteqn{\left(\om\d_t+a_j(x,\la)\d_x\right)\int_\xi^x\frac{c_j(\eta,x,\la)}{a_j(\eta,\la)}f_j(\eta,\la,u(\eta,\tau_j(\eta,x,t,\la,\om)))d\eta}\nonumber\\
&&=f_j(x,\la,u(x,t))-\partial_{u_j}b_j(x,\la,0)\int_\xi^x\frac{c_j(\eta,x,\la)}{a_j(\eta,\la)}f_j(\eta,\la,u(\eta,\tau_j(\eta,x,t,\la,\om)))d\eta.
\label{72}
\end{eqnarray}
Therefore, if $u$ satisfies \reff{rep1}--\reff{rep2}, then it satisfies \reff{eq:1.1}.

Now, for $\la \in[-\delta_0,\delta_0]$ and $\om \in \R$ we define linear bounded operators $C(\la, \om): \CC_n \to \CC_n$ by
\beq
\label{Cdef}
\left(C(\la,\om)u\right)_j(x,t):=
\left\{
\begin{array}{ll}
\displaystyle
c_j(0,x,\la)\sum_{k=m+1}^nr_{jk}u_k(0,\tau_j(0,x,t,\la,\om)), \;j=1,\ldots,m,\\
\displaystyle
c_j(1,x,\la)\sum_{k=1}^mr_{jk}u_k(1,\tau_j(1,x,t,\la,\om)), \;j=m+1,\ldots,n
\end{array}
\right.
\ee
and nonlinear operators  $F(\la, \om,\cdot): \CC_n \to \CC_n$ by
\beq
\label{Fdef}
F(\la,\om,u)_j(x,t):=
\int_{x_j}^x \frac{c_j(\xi,x,\la)}{a_j(\xi,\la)}f_j(\xi,\lambda,u(\xi,\tau_j(\xi,x,t,\lambda,\om)))d\xi.
\ee
Here and in what follows we denote
$$
x_j:=\left\{
\begin{array}{rcl}
0 & \mbox{for} & j=1,\dots,m,\\
1 & \mbox{for} & j=m+1,\dots,n.
\end{array}\nonumber
\right.
$$
Using this notation, system \reff{rep1}--\reff{rep2} is equivalent to the operator equation
\beq
\label{abstract}
u=C(\la,\om)u+F(\la,\om,u).
\ee

Because of assumption \reff{smooth} and well-known differentiability properties of Nemytskii operators (see, e.g., \cite[Lemma 2.4.18]{Abraham},
\cite[Lemma 6.1]{Amann})
the superposition operator
$
u \in \CC_n \mapsto f(\cdot,\la,u(\cdot))\in \CC_n
$
(cf. \reff{fdef})
is $C^\infty$-smooth. Hence, $F(\la,\om,\cdot)$ is $C^\infty$-smooth. It is easy to verify that
\beq
\label{cont1}
(\la,\om,u) \in [-\delta_0,\delta_0] \times \R \times \CC_n \mapsto C(\la,\om)u \in \CC_n \mbox{ is continuous}
\ee
and that for all $k=0,1,2,\ldots$
\beq
\label{cont2}
\begin{array}{l}
(\la,\om,u,v_1,\ldots,v_k) \in [-\delta_0,\delta_0] \times \R \times \CC_n\times \ldots \times \CC_n  \nonumber\\
\mapsto \partial^k_u
F(\la,\om,u)(v_1,\ldots,v_k) \in \CC_n \mbox{ is continuous.}
\end{array}
\ee
Moreover, for all  $k=0,1,2,\ldots$ and $\rho>0$ there exists $c_k(\rho)>0$ such that for all $\la \in [-\delta_0,\delta_0]$, $\om \in \R$, $u \in \CC_n$  
with $\|u\|_\infty \le \rho$ and $v_1,\ldots,v_k \in  \CC_n$ it holds 
\beq
\label{cont3}
\|C(\la,\om)v_1 \|_\infty \le c_1(\rho) \|v_1\|_\infty,\quad 
\| \partial^k_u
F(\la,\om,u)(v_1,\ldots,v_k)\|_\infty \le c_k(\rho) \|v_1\|_\infty \ldots \|v_k\|_\infty.
\ee
\begin{rem}
Unfortunately,  the maps $(\la,\om) \in [-\delta_0,\delta_0]
\times \R \mapsto C(\la,\om)u \in  \CC_n$ as well as  
$(\la,\om) \in [-\delta_0,\delta_0]\times \R \mapsto F(\la,\om,u) \in  \CC_n$ 
are not smooth, in general,
if $u$ is only continuous and not smooth, cf. \reff{lader}.
This makes the question, if the data-to-solution map corresponding to \reff{abstract} 
is smooth, very delicate.
\end{rem}

\begin{defn}\label{C^1spaces}
We denote by $\CC_n^1$  the Banach space of all $u \in \CC_n$ such that the partial derivatives 
$\d_xu$ and $\d_tu$
exist and are continuous with the norm $\|u\|_\infty +\|\d_xu\|_\infty +\|\d_tu\|_\infty$.
\end{defn}

Directly from the definitions \reff{Cdef} and \reff{Fdef} it follows that for all $u \in \CC_n^1$ we have
$C(\la,\om)u \in \CC_n^1$, $F(\la,\om,u) \in \CC_n^1$, and
\beq
\label{Abl}
\begin{array}{rcl}
\d_xC(\la,\om)u&=&C^x_0(\la,\om)u+C^x_1(\la,\om)\d_tu,\\
\d_tC(\la,\om)u&=&C(\la,\om)\d_tu,\\
\d_xF(\la,\om,u)&=&F^x_0(\la,\om,u)+F^x_1(\la,\om,u)\d_tu,\\
\d_tF(\la,\om,u)&=&\d_uF(\la,\om,u)\d_tu,
\end{array}
\ee
where  $C^x_0(\la,\om)$,  $C^x_1(\la,\om)$, and  $F^x_1(\la,\om,u)$ are linear bounded operators from $\CC_n$ into  $\CC_n$
defined by
$$
\begin{array}{rcl}
\displaystyle
\left(C^x_0(\la,\om)u\right)_j(x,t)&:=&
\left\{
\begin{array}{ll}
\displaystyle
\d_xc_j(0,x,\la)\sum_{k=m+1}^nr_{jk}u_k(0,\tau_j(0,x,t,\la,\om)), \;j=1,\ldots,m,\\
\displaystyle
\d_xc_j(1,x,\la)\sum_{k=1}^mr_{jk}u_k(1,\tau_j(1,x,t,\la,\om)), \;j=m+1,\ldots,n,
\end{array}
\right.\\
\displaystyle
\left(C^x_1(\la,\om)v\right)_j(x,t)&:=&
\left\{
\begin{array}{ll}
\displaystyle
-\frac{\om c_j(0,x,\la)}{a_j(x,\la)}\sum_{k=m+1}^nr_{jk}v_k(0,\tau_j(0,x,t,\la,\om)), \;j=1,\ldots,m,\\
\displaystyle
-\frac{\om c_j(1,x,\la)}{a_j(x,\la)}\sum_{k=1}^mr_{jk}v_k(1,\tau_j(1,x,t,\la,\om)), \;j=m+1,\ldots,n,
\end{array}
\right.\\
\end{array}
$$
and
\begin{eqnarray*}
&&(F^x_1(\la,\om,u)v)_j(x,t)\\
&&:=
\displaystyle-\int_{x_j}^x\frac{\om c_j(\xi,x,\la)}{a_j(x,\la)a_j(\xi,\la)}\sum_{k=1}^n
\d_{u_k}f_j(\xi,\la,u(\xi,\tau_j(\xi,x,t,\la,\om)))
v_k(\xi,\tau_j(\xi,x,t,\la,\om))d\xi,
\end{eqnarray*}
and $F_0^x(\la,\om,u) \in \CC_n$ is defined by
$$
(F^x_0(\la,\om,u))_j(x,t):=
\frac{f_j(x,\la,u(x,t))}{a_j(x,t)}
+\int_{x_j}^x\frac{\d_xc_j(\xi,x,\la)}{a_j(\xi,\la)}f_j(\xi,\la,u(\xi,\tau_j(\xi,x,t,\la,\om)))d\xi.
$$

\begin{lemma}\label{Fredho}
Suppose \reff{Fred}. Then for any $\ga_1\in(0,1)$ there exist $\delta_1 \in (0,\delta_0]$ and $c_1>0$ 
such that for all $\la \in [-\delta_1,\delta_1]$ and all 
$\om \in [1-\ga_1,1+\ga_1]$
the following is true:

(i) The operators  $I-C(\la, \om)$ are isomorphisms  on $\CC_n$  as well as   on $\CC_n^1$, and 
$$
\|(I-C(\la,\om))^{-1}\|_{\LL(\CC_n)}+\|(I-C(\la,\om))^{-1}\|_{\LL(\CC_n^1)} \le c_1.
$$

(ii) The operators
$I-C(\la, \om) -\partial_uF(\la,\om,0)$ are Fredholm operators of index zero  from $\CC_n$ into $\CC_n$.
\end{lemma}
\begin{proof} 
To prove assertion (i), denote by $\CC_m$ the space of all continuous maps 
$v: [0,1] \times \R \to \R^m$ with $v(x,t+2\pi)=v(x,t)$ for all $x \in [0,1]$ and $t \in \R$, with the norm
$$
\|v\|_\infty:=\max_{1 \le j \le m} \max_{0 \le x \le 1} \max_{t \in \R} |v_j(x,t)|.
$$
Similarly we define the space $\CC_{n-m}$. The spaces $\CC_n$ and $\CC_m \times 
\CC_{n-m}$ will be identified, i.e., elements $u \in \CC_n$ will be written as $u=(v,w)$ with $v \in \CC_m$ and $w \in \CC_{n-m}$.
Then the operators $C(\la,\om)$ work as
\beq
\label{KL}
C(\la,\om)u=(K(\la,\om)w,L(\la,\om)v)\;\; \mbox{ for } u=(v,w),
\ee
where the linear bounded operators $K(\la,\om):\CC_{n-m} \to \CC_m$ and  $L(\la,\om):\CC_{m} \to \CC_{n-m}$
are defined by the right hand side of \reff{Cdef}.

Let $f=(g,h) \in \CC_n$ with $g \in \CC_m$ and $h \in \CC_{n-m}$ be arbitrarily given.
We have $u=C(\la,\om)u+f$  
if and only if
$
v=K(\la,\om)w+g, \; w=L(\la,\om)v+h,
$
i.e., if and only if
\beq
\label{inserted}
v=K(\la,\om)(L(\la,\om)v+h)+g, \quad w=L(\la,\om)v+h.
\ee
Moreover, it holds
$$
\|K(\la,\om)w\|_\infty=\max_{1 \le j \le m} \max_{0 \le x \le 1} 
\max_{t \in \R}\left|c_j(0,x,\la)\sum_{k=m+1}^nr_{jk}w_k(0,\tau_j(0,x,t,\la,\om))\right|
\le   R_0(\la) \|w\|_\infty 
$$
with
$$
R_0(\la):= 
\max_{1 \le j \le m} \max_{0 \le x \le 1}
\sum_{k=m+1}^n|r_{jk}| \exp \left(-\int_0^x\frac{\d_{u_j}b_j(\eta,\la,0)}{a_j(\eta,\la)}d\eta\right).
$$
Similarly one shows that $\|L(\la,\om)v\|_\infty\le R_1(\la)\|v\|_\infty$ with
$$
R_1(\la):= 
\max_{m+1 \le j \le n} \max_{0 \le x \le 1} 
\sum_{k=m+1}^n|r_{jk}|\exp \int_x^1\frac{\d_{u_j}b_j(\eta,\la,0)}{a_j(\eta,\la)}d\eta.
$$
Since $R_0(\la)$ and $R_1(\la)$ depend continuously on $\la$, notation
 \reff{Rdef0} and \reff{Rdef1} and 
assumption \reff{Fred} yield that there exist  $\delta_1 \in (0,\delta_0]$ and $c \in(0,1)$ such that for all $\la \in [-\delta_1,\delta_1]$
and all $\om \in \R$ we have
$$
\|K(\la,\om)\|_{\LL(\CC_{n-m};\CC_m)}\|L(\la,\om)\|_{\LL(\CC_{m};\CC_{n-m})} \le c.
$$
Therefore, for those $\la$ and $\om$ it holds
$$
\|(I-K(\la,\om)L(\la,\om))^{-1}\|_{\LL(\CC_{n-m})} \le \frac{1}{1-c},
$$
hence, the system \reff{inserted} is uniquely solvable with respect to $v$ and $w$.
Moreover, we have the following a priori estimates:
\begin{eqnarray*}
&&\|v\|_\infty \le  \frac{1}{1-c} \|K(\la,\om)h+g\|_\infty \le
 \frac{1}{1-c} \left(R_0(\la)\|h\|_\infty+\|g\|_\infty\right),\\
&&\|w\|_\infty =  \|L(\la,\om)v+h\|_\infty \le
 \mbox{const }\left(\|g\|_\infty+\|h\|_\infty\right),
\end{eqnarray*}
i.e., $\|u\|_\infty \le \mbox{const} \,\|f\|_\infty$.
Here the constants do not depend on $\la \in [-\delta_1,\delta_1]$ and $\om \in \R$.

Now fix an arbitrary $\ga_1>0$. 
Suppose $u=C(\la,\om)u+f$ with $f \in \CC_n^1$, $\la \in[-\delta_1,\delta_1]$, and $\om \in [-\ga_1,\ga_1]$.
Then  \cite[Theorem 1.2 (iv)]{KR2}
implies that  $u \in \CC_n^1$.
Hence, it remains to show that
\beq
\label{apriori}
\|u\|_\infty+\|\d_xu\|_\infty+\|\d_tu\|_\infty \le \mbox{const}\left(\|f\|_\infty+\|\d_xf\|_\infty+\|\d_tf\|_\infty\right),
\ee
where the constant can be chosen independently of $u$, $f$, $\la$, and $\om$.

Because of 
\reff{Abl} we have
\begin{eqnarray}
\label{1}
&&\d_xu=C_0^x(\la,\om)u+C_1^x(\la,\om)\d_tu+\d_xf,\\
\label{2}
&&\d_tu=C(\la,\om)\d_tu+\d_tf.
\end{eqnarray}
The equation  \reff{2} now yields $\|\d_tu\|_\infty \le \mbox{const }\|\d_tf\|_\infty$, and this together with \reff{1} gives 
\reff{apriori}.
Here the constants do not depend on $\la \in [-\delta_1,\delta_1]$ and $\om \in [-\gamma_1,\gamma_1]$, but on $\gamma_1$, 
because the norm of the operator $C^x_1(\la,\om)$ grows if $\om$ grows, in general.

To prove assertion (ii), we take into account \reff{fdef} and \reff{Fdef}, hence 
\beq
\label{dF1}
\left(\d_uF(\la,\om,0)u\right)_j(x,t)
=-\int_{x_j}^x \frac{c_j(\xi,x,\la)}{a_j(\xi,\la)}\sum_{k=1\atop k\not=j}^n\d_{u_k}b_j(\xi,\la,0)
u_k(\xi,\tau_j(\xi,x,t,\lambda,\om)))d\xi.
\ee
Since  the right-hand side of \reff{dF1} does not depend on $u_j$,
we can use assertions (i)--(iii) of \cite[Theorem 1.2]{KR3} and state 
that the operators
$I-C(\la, \om) -\partial_uF(\la,\om,0)$ are Fredholm of index zero  from $\CC_n$ into $\CC_n$ if
$R_0(\la) R_1(\la) <1$.
Because of  assumption  \reff{Fred} this is the case if $\la$ is sufficiently close to zero.
\end{proof}

\section{Liapunov-Schmidt procedure}
\label{Lyapunov-Schmidt Procedure}
\renewcommand{\theequation}{{\thesection}.\arabic{equation}}
\setcounter{equation}{0}

In this section we do a {Liapunov-Schmidt procedure in order to reduce (locally for $\la \approx 0, \om \approx 1$ and $u \approx 0$)
the problem \reff{abstract} with infinite-dimensional state parameter $(\om,u) \in \R \times \CC_n$ to a problem with two-dimensional state parameter.

For $\la\in [-\delta_0,\delta_0]$ and  $\om\in\R$ we introduce linear bounded operators $A(\la,\om),\tilde{A}(\la,\om): \CC_n^1 \to \CC_n$ by
\begin{eqnarray*}
\left[A(\la,\om)u\right](x,t)&:=&\left[\om\d_tu_j(x,t)+a_j(x,\la)\d_xu_j(x,t)+\d_{u_j}b_j(x,\la,0)u_j(x,t)\right]_{j=1}^n,\\
\left[\tilde{A}(\la,\om)u\right](x,t)&:=&\left[-\om\d_tu_j(x,t)-\d_x(a_j(x,\la)u_j(x,t))+\d_{u_j}b_j(x,\la,0)u_j(x,t)\right]_{j=1}^n
\end{eqnarray*}
and linear bounded operators $B(\la),\tilde{B}(\la),D(\la):\CC_n \to \CC_n$ by
\begin{eqnarray*}
\left[B(\la)u\right](x,t)&:=&\left[\sum_{k=1\atop k\not=j}^n\d_{u_k}b_j(x,\la,0)u_k(x,t)\right]_{j=1}^n,\\
\left[\tilde{B}(\la)u\right](x,t)&:=&\left[\sum_{k=1\atop k\not=j}^n\d_{u_j}b_k(x,\la,0)u_k(x,t)\right]_{j=1}^n,\\
\left[D(\la,\om)u\right](x,t)&:=&\left[\int_{x_j}^x \frac{c_j(\xi,x,\la)}{a_j(\xi,0)}u_j(\xi,\tau_j(\xi,x,t,\la,\om))d\xi\right]_{j=1}^n.
\end{eqnarray*}
Finally, we denote by
\beq
\label{sp}
\langle u,v \rangle := \frac{1}{2\pi}\sum_{j=1}^n\int_0^{2\pi} \int_0^1 u_j(x,t)v_j(x,t) dxdt
\ee
the $L^2$ scalar product in $\CC_n$. Obviously, for all  $u,v \in \CC_n^1$ it holds
\beq
\label{AtildeA}
\langle A(\la,\om)u,v \rangle-\langle u,\tilde{A}(\la,\om)v \rangle=\left[\sum_{j=1}^na_j(x,\la)\int_0^{2\pi}u_j(x,t)v_j(x,t)dt\;\right]_{x=0}^{x=1},
\ee  
and  for all  $u,v \in \CC_n$ we have
\beq
\label{BtildeB}
\langle B(\la)u,v \rangle-\langle u,\tilde{B}(\la)v \rangle=0.
\ee
In particular, if $u$ satisfies the boundary conditions \reff{eq:1.2} and $v$  satisfies the adjoint boundary conditions
\beq\label{adbc}
\begin{array}{l}
\displaystyle
a_j(0,\la)v_j(0,t) = -\sum\limits_{k=1}^mr_{kj}a_k(0,\la)v_k(0,t),\quad  j=m+1,\ldots,n,\\
\displaystyle
a_j(1,\la)v_j(1,t) = -\sum\limits_{k=m+1}^nr_{kj}a_k(1,\la)v_k(1,t), \quad  j=m+1,\ldots,n,
\end{array}
\ee
then $\langle A(\la,\om)u,v \rangle=\langle u,\tilde{A}(\la,\om)v \rangle$.

\begin{lemma}\label{AB} (i) For all 
$u \in \CC_n^1$ it holds $A(\la,\om)C(\la,\om)u=0$,   $A(\la,\om)D(\la,\om)u=u$,
$A(\la,\om)\d_uF(\la,\om,0)u=-B(\la)u$, and
\beq
\label{Af}
\left[A(\la,\om)F(\la,\om,u)\right](x,t)=f(x,\la,u(x,t)).
\ee

(ii)  For all 
$u \in \CC_n^1$ satisfying the boundary conditions \reff{eq:1.2} it holds
$D(\la,\om)A(\la,\om)u=(I-C(\la,\om))u$.
\end{lemma}
\begin{proof}
(i) Take $u \in \CC^1_n$. Then $A(\la,\om)C(\la,\om)u=0$ follows from \reff{71} (taking there $\xi=x_j$),
$A(\la,\om)D(\la,\om)u=u$  follows from \reff{72} (taking there $\xi=x_j$ and $f_j(x,\la,u)=u_j$). Similarly, \reff{Af}
follows from \reff{72} (taking there $\xi=x_j$).
Finally, from \reff{dF1} we have $\d_uF(\la,\om,0)u=-D(\la,\om)B(\la)u$, hence  $B(\la)u=-A(\la,\om)\d_uF(\la,\om,0)u$.

(ii) In Section \ref{Function Spaces and Operators} we showed
the following (cf. \reff{Deq}): 
If we have $A(\la,\om)u=f$ for some  $f \in \CC_n$ and some $u \in \CC_n^1$ which satisfies the boundary conditions \reff{eq:1.2},
then $(I-C(\la,\om))u=D(\la,\om)f$. 
\end{proof}

\subsection{Kernel and image of the linearization}

Using the functions $v^0,w^0: [0,1] \to \C^n$, 
introduced in Section \ref{sec:results} (see \reff{normed}), we define functions
${\bf v,w}: [0,1]\times \R \to \C^n$ and ${\bf v_1,v_2,w_1,w_2}: [0,1]\times \R \to \R^n$
by
\beq\label{def}
\begin{array}{l}
{\bf v}(x,t):=v^0(x)e^{-it}, \; {\bf w}(x,t):=w^0(x)e^{-it},\\
{\bf v_1}:=\Re {\bf v},\;  {\bf v_2}:=\Im {\bf v},\;  {\bf w_1}:=\Re {\bf w},\; {\bf w_2}:=\Im {\bf w}.
\end{array}
\ee
It follows from \reff{normed} and \reff{sp} that
\beq
\label{Kro}
\langle {\bf v_j},{\bf w_k} \rangle =\delta_{jk}.
\ee 
Remark that here we used the number two in the normalization condition  \reff{normed}.
Further, we define a linear bounded operator  $L_0: \CC_n \to \CC_n$ by
$$
L_0:=I-C(0,1)-\d_uF(0,1,0).
$$ 
Due to Lemma~\ref{Fredho} (ii),  $L_0$ is a Fredholm operator of index zero from $\CC_n$ into $\CC_n$. 
To simplify further notation we will  write
\beq
\label{ABdef}
A:=A(0,1),\; \tilde A:=\tilde A(0,1),\;
B:=B(0), \;\tilde B:=\tilde B(0).
\ee

\begin{lemma}\label{kerim}
We have
$$
\ker L_0=\mbox{\rm span} \,\{{\bf v_1},{\bf v_2}\},\;\;
\im L_0 =\left\{f \in \CC_n:\; \langle f, \tilde{A}{\bf w_1} \rangle= \langle f, \tilde{A}{\bf w_2} \rangle=0\right\}.
$$
\end{lemma}

\begin{proof}
Take $u \in \ker L_0$. Then, by \cite[Theorem 1.2(iv)]{KR2} we have $u \in \CC_n^1$.
Hence, because of Lemma \ref{AB} it holds
$
AL_0u=A(I-C(0,1)-\d_uF(0,0))u=(A+B)u=0,
$
i.e.,
$$
\partial_tu_j(x,t)  + a_j(x,0)\partial_xu_j(x,t) + \d_{u_j}b_j(x,0,0)u_j(x,t)  = 0.
$$
Moreover, $u$ satisfies the boundary conditions \reff{eq:1.2} because any function of the type
$C(\la,\om)v+\d_uF(\la,\om,0)v$ with arbitrary $v \in \CC_n$ satisfies those boundary conditions.
Doing the Fourier ansatz
$$
u(x,t)=\sum_{s \in \Z}u^s(x)e^{ist},
$$
we get the following boundary value problem for the coefficient $u^s$:
$$
\begin{array}{rcll}
\displaystyle
a_j(x,0)\frac{d}{dx}u_j^s(x)+\sum_{k=1}^n\d_{u_k}b_{j}(x,0,0)u_k^s(x) &=& -isu_j^s(x) ,& j=1,\ldots,n,\\
\displaystyle
u_j^s(0) &=& \displaystyle\sum\limits_{k=m+1}^nr_{jk}u_k^s(0),& j=1,\ldots,m,\\
\displaystyle
u_j^s(1) &=& \displaystyle\sum\limits_{k=1}^mr_{jk}u_k^s(1),& j=m+1,\ldots,n.
\end{array}
$$ 
By assumptions \reff{geo} and \reff{nonres}, this is equivalent to
$$
u^s=0 \mbox{ for all } s \in \Z \setminus \{-1,1\} \mbox{ and } u^{\pm 1}\in \mbox{span }\{\Re v^0, \Im v^0\},
$$
i.e., to $u \in \mbox{\rm span} \,\{\Re {\bf v},\Im {\bf v}\}$.
In particular, we have $(A+B){\bf v}=0$.
Similarly one shows that  
\beq
\label{tilde}
(\tilde{A}+\tilde{B}){\bf w}=0
\ee 
and that ${\bf w}$ satisfies the adjoint boundary conditions \reff{adbc}.

Now we show that $\im L_0 \subseteq\{f \in \CC_n:\; \langle f, \tilde{A}{\bf w_1} \rangle= \langle f, \tilde{A}{\bf w_2} \rangle=0\}$:
Take $f \in \im L_0$, i.e., $f=(I-C(0,1)-\d_uF(0,0))u$ with arbitrary $u \in \CC_n$. 
Using Lemma \ref{AB}, \reff{AtildeA}, the fact that  any function of the type
$C(\la,\om)v+\d_uF(\la,\om,0)v$ with a certain $v \in \CC_n$ satisfies \reff{eq:1.2}
and that ${\bf w}$ satisfies \reff{adbc}, we get
$$
\langle(C(0,1)+\d_uF(0,1,0))u,\tilde{A} {\bf w_j} \rangle=
\langle A(C(0,1)+\d_uF(0,1,0))u,{\bf w_j} \rangle=
-\langle Bu,{\bf w_j} \rangle=
-\langle u,\tilde{B}{\bf w_j} \rangle.
$$
Hence, \reff{tilde} yields
$
\langle f,\tilde{A} {\bf w_j} \rangle=\langle(I-C(0,1)-\d_uF(0,1,0))u,\tilde{A} {\bf w_j} \rangle=
\langle u,(\tilde{A}+\tilde{B}){\bf w_j} \rangle=0.
$

Finally we show that $\{f \in \CC_n:\; \langle f, \tilde{A}{\bf w_1} \rangle= \langle f, \tilde{A}{\bf w_2} 
\rangle=0\}\subseteq \im L_0$:
Because of Lemma \ref{Fredho} (i) there exist uniquely defined  functions ${\bf \tilde{v}_1}, {\bf \tilde{v}_2} \in \CC_n^1$
such that
\beq
\label{vtilde}
(I-C(0,1)){\bf \tilde{v}_k}=D{\bf v_k}, \; k=1,2.
\ee
Moreover, these functions satisfy the boundary conditions  \reff{adbc}. Therefore,
\beq
\label{deltajk}
\langle  {\bf \tilde{v}_k},\tilde{A}{\bf w_l} \rangle=\langle  A{\bf \tilde{v}_k},{\bf w_l} \rangle=
\langle  {\bf v_k},{\bf w_l} \rangle =\delta_{kl}
\ee
and, hence,
$
\dim \{f \in \CC_n:\; \langle f, \tilde{A}{\bf w_1} \rangle= \langle f, \tilde{A}{\bf w_2} \rangle=0\} \ge 2.
$
But $\im L_0$ is a closed subspace of codimension two in $\CC_n$ because of Lemma~\ref{Fredho}, therefore the claim follows.
\end{proof}

\begin{rem}
The case $m=0$ (and, similarly, $m=n$) is  not of interest for  Hopf bifurcation analysis. The 
reason is that 
it does not fit  \reff{geo}, what is one of the crucial standard assumptions in Hopf bifurcation
theorems. One can easily check that $\dim\ker L_0=0$ in this case. 
Then, by Lemma \ref{kerim}, the problem \reff{evp} with $\la=0$ does not have a pure imaginary pair of 
geometrically simple eigenvalues, what means that the assumption  \reff{geo} fails to be fulfilled.
\end{rem}

\subsection{Projectors and splitting of \reff{abstract}}

Lemma \ref{kerim} and \reff{deltajk} imply
that the linear bounded operator $P: \CC_n \to \CC_n$, which is defined by
\beq
\label{Pdef}
Pu:=\sum_{k=1}^2 \langle u,\tilde{A} {\bf w_k} \rangle {\bf \tilde{v}_k}
\ee
is a projection with $\ker P= \im L_0$, where the functions  ${\bf \tilde{v}_k}$
are implicitly defined in \reff{vtilde}.
Similarly, the 
linear bounded operator $Q: \CC_n \to \CC_n$, which is defined by
\beq
\label{Qdef}
Qu:=\sum_{k=1}^2 \langle u,{\bf w_k} \rangle {\bf v_k}
\ee
is a projection with $\im Q= \ker L_0$.  

Now we are going to solve the equation \reff{abstract} by means of the ansatz
$$
u=v+w, \; v \in \ker L_0 = \im Q, \; w \in \ker Q.
$$
Hence, we have to solve a coupled system consisting of the finite dimensional equation
\beq\label{finite}
P((I-C(\la,\om))(v+w)-F(\la,\om,v+w))=0
\ee
and the infinite dimensional equation
\beq\label{infinite}
(I-P)((I-C(\la,\om))(v+w)-F(\la,\om,v+w))=0.
\ee

\subsection{Local solution of \reff{infinite}}\label{sec:infinite}

In this subsection we will solve \reff{infinite} locally with respect to $w \approx 0$ for parameters 
$\om \approx 1$, $\la \approx 0$, and $v \approx 0$. Unfortunately, the classical implicit function theorem
cannot be used for that purpose because the left-hand side of  \reff{infinite} is not $C^1$-smooth. More exactly, the map
$(\la,\om) \in \R^2 \mapsto (I-P)(I-C(\la,\om)-\d_uF(\la,\om,0)) \in \LL(\CC_n)$ is not continuous.
The reason is that for operators of ``shift type'' like $(S_\tau u)(x,t):=u(x,t+\tau)$
the map $\tau \in \R \mapsto S_\tau \in \LL(\CC_n)$ is not continuous (with respect to the operator norm in $\LL(\CC_n)$).

There exist several generalizations of the  implicit function theorem in which the map ``control parameter $\mapsto$ 
linearization with respect to the state parameter'' is allowed to  be discontinuous with respect to the operator norm, see, e.g.,
\cite[Theorem 7]{Appell}, \cite[Theorem 2.1]{Renardy1}.
However, it turns out that they do not fit to our problem, so we are going to adapt ideas of  \cite{Magnus}
and \cite[Theorem 2.1]{RO}.

\begin{lemma}
\label{Magnus}
There exist $\delta_2 \in(0,\delta_1)$ and $c_2>0$ such that for all $\la \in [-\delta_2,\delta_2]$, $\om \in [1-\delta_2,1+\delta_2]$ 
and $u \in \CC_n$ with $\|u\|_\infty \le \delta_2$ it holds
$$
\|(I-P)(I-C(\la,\om)-\d_uF(\la,\om,u))w\|_\infty \ge c_2 \|w\|_\infty \mbox{ for all } w \in \ker Q.
$$
\end{lemma}
\begin{proof}
Suppose the contrary. Then there exist sequences $\la^1,\la^2,\ldots \in \R$ with $\la^r \to 0$; $\om^1,\om^2,\ldots \in \R$
with $\om^r \to 1$; $u^1,u^2,\ldots \in \CC_n$ with $\|u^r\|_\infty \to 0$ and $w^1,w^2,\ldots \in  \ker Q$ with
\beq
\label{w=1}
\|w^r\|_\infty=1
\ee
and
\beq
\label{Fto0}
\|(I-P)(I-C(\la^r,\om^r)-\d_uF(\la^r,\om^r,u^r))w^r\|_\infty \to 0.
\ee

We intend to show that 
a subsequence of $\left\{w^r: \; r \in \N\right\}$ converges to zero  in $\CC_n$,
 getting a contradiction to \reff{w=1}.
The proof is divided into a sequence of claims.

\begin{claim} \label{cl:1} The convergence \reff{Fto0}  holds with $u^1=u^2=\ldots=0$, namely
$$
\|(I-P)(I-C(\la^r,\om^r)-\d_uF(\la^r,\om^r,0))w^r\|_\infty \to 0.
$$
\end{claim}
\begin{subproof}
Because of \reff{cont3} we have 
$
\|(\d_uF(\la^r,\om^r,u^r)-\d_uF(\la^r,\om^r,0))w^r\|_\infty \le \mbox{const} \|u^r\|_\infty,
$
where the constant does not depend on $r$. Hence, the assumptions $\|u^r\|_\infty \to 0$ and \reff{Fto0} yield the claim.
\end{subproof}

\begin{claim} \label{cl:2} The sequence  
$
\left(C(\la^r,\om^r)-C(0,1)\right)w^r+\left(\d_uF(\la^r,\om^r,0)-\d_uF(0,1,0)\right)w^r 
$
converges weakly to zero in $L^2\left((0,1)\times(0,2\pi);\R^n\right)$.
\end{claim}
\begin{subproof}
Denote by $t_j(\xi,x,\cdot,\la,\om)$ the inverse function to the function $\tau_j(\xi,x,\cdot,\la,\om)$, i.e.,
$$
t_j(\xi,x,\tau,\la,\om)=\tau-\om\int_x^\xi\frac{d \eta}{a_j(\eta,\la)}
$$
(cf. \reff{charom}). Take a test function $\vphi \in \CC_n^1$. Then we have
\begin{eqnarray*}
\lefteqn{
\int_0^1\int_0^{2\pi}c_j(x_j,x,\la^r)w_k^r(x_j,\tau_j(x_j,x,t,\la^r,\om^r))\vphi_j(x,t)dtdx}\\
&&-\int_0^1\int_0^{2\pi}c_j(x_j,x,0)w_k^r(x_j,\tau_j(x_j,x,t,0,1))\vphi_j(x,t)dtdx\\
&& 
=\int_0^1\int_0^{2\pi}c_j(x_j,x,\la^r)\vphi_j(x,t_j(x_j,x,\tau,\la^r,\om^r))w_k^r(x_j,\tau) d\tau dx\\
&&-\int_0^1\int_0^{2\pi}c_j(x_j,x,0)\vphi_j(x,t_j(x_j,x,\tau,0,1)) 
w_k^r(x_j,\tau) d\tau dx,
\end{eqnarray*}
and this tends to zero for $r \to \infty$. So we get
$
\langle \left(C(\la^r,\om^r)-C(0,1)\right)w^r,\vphi \rangle \to 0.
$
Similarly we have
\begin{eqnarray*}
\lefteqn{
\int_0^1\int_{x_j}^x\int_0^{2\pi}\left(\frac{c_j(\xi,x,\la^r)}{a_j(\xi,\la^r)}
\d_{u_k}f_j(\xi,\la^r,0)
w_k^r(\xi,\tau_j(\xi,x,t,\la^r,\om^r))\right.}\\
&&
\left.-\frac{c_j(\xi,x,0)}{a_j(\xi,0)}\d_{u_k}
f_j(\xi,0,0)w_k^r(\xi,\tau_j(\xi,x,t,0,1))\right)\vphi_j(x,t)dt d\xi dx\\
&&
=\int_0^1\int_{x_j}^x\int_0^{2\pi}\left(\frac{c_j(\xi,x,\la^r)}{a_j(\xi,\la^r)}
\d_{u_k}f_j(\xi,\la^r,0)
\vphi_j(x,t_j(\xi,x,\tau,\la^r,\om^r))\right.\\
&&\left.-\frac{c_j(\xi,x,0)}{a_j(\xi,0)}\d_{u_k}f_j(\xi,0,0)
\vphi_j(x,t_j(\xi,x,\tau,0,1))\right)w_k^r(x,\tau)d \tau d\xi dx,
\end{eqnarray*}
and this tends to zero for $r \to \infty$. Therefore,
$
\langle \left(\d_uF(\la^r,\om^r,0)-\d_uF(0,1,0)\right)w^r,\vphi \rangle 
$ 
tends to zero.
\end{subproof}

\begin{claim} \label{cl:5} The operators $\d_uF(\la^r,\om^r,0)^2$ and 
$\d_uF(\la^r,\om^r,0)C(\la^r,\om^r)$ 
map continuously
$\CC_n$  into $\CC_n^1$ and 
\begin{eqnarray*}
\|\d_uF(\la^r,\om^r,0)^2\|_{ \LL(\CC_n;\CC_n^1)}+
\|\d_uF(\la^r,\om^r,0)C(\la^r,\om^r)\|_{ \LL(\CC_n;\CC_n^1)}\le \mbox{const}.
\end{eqnarray*}
\end{claim}
\begin{subproof}
We first give the proof for the operators $\d_uF(\la^r,\om^r,0)^2$. 
Because of \reff{Abl} we have for all $u \in \CC_n^1$
\beq
\label{du}
\begin{array}{rcl}
\d_x\d_uF(\la^r,\om^r,0)u&=&G^ru+H^r\d_tu,\\
\d_t\d_uF(\la^r,\om^r,0)u&=&\d_uF(\la^r,\om^r,0)\d_tu,
\end{array}
\ee
where the linear bounded operators $G^r,H^r: \CC_n \to \CC_n$ are defined 
by $G^r:=\d_uF_0^x(\la^r,\om^r,0)$ and $H^r:=F_1^x(\la^r,\om^r,0)$.
Therefore,
\begin{eqnarray*}
&&\d_t\d_uF(\la^r,\om^r,0)^2u=\d_uF(\la^r,\om^r,0)^2\d_tu, \\
&&\d_x\d_uF(\la^r,\om^r,0)^2u=G^r\d_uF(\la^r,\om^r,0)u+H^r\d_uF(\la^r,\om^r,0)\d_tu.
\end{eqnarray*}
hence,
\begin{eqnarray*}
\lefteqn{
\|\d_t\d_uF(\la^r,\om^r,0)^2u\|_\infty+\|\d_x\d_uF(\la^r,\om^r,0)^2u\|_\infty}\\
&& \le \mbox{const}(\|u\|_\infty+\|\d_uF(\la^r,\om^r,0)^2 \d_tu\|_\infty+
\|H^r\d_uF(\la^r,\om^r,0)\d_tu\|_\infty).
\end{eqnarray*}
Here we used \reff{cont3} and the fact that the operators $G^r$ 
and $H^r$ are uniformly bounded with respect to $r$ in the uniform operator norm.
Hence, 
we have to show that 
\beq
\label{subseq4}
\|\d_uF(\la^r,\om^r,0)^2\d_tu\|_\infty+\|H^r\d_uF(\la^r,\om^r,0)\d_tu\|_\infty\le \mbox{const}\|u\|_\infty
\mbox{ for all } u \in \CC_n^1.
\ee

Let us start with  $\d_uF(\la^r,\om^r,0)^2 \d_tu$. Because of \reff{dF1} we have
\begin{eqnarray}
\lefteqn{
(\d_uF(\la^r,\om^r,0)^2 \d_tu)_j(x,t)}\nonumber\\
&&=
\int_{x_j}^x \frac{c_j(\xi,x,\la^r)}{a_j(\xi,\la^r)}\sum_{k=1\atop k\not=j}^n\d_{u_k}b_j(\xi,\la^r,0)
v_k^r(\xi,\tau_j(\xi,x,t,\lambda^r,\om^r)))d\xi
\label{int1}
\end{eqnarray}
with
\beq
\label{int2}
v_k^r(\xi,\tau):=
\int_{x_j}^\xi \frac{c_k(\eta,\xi,\la^r)}{a_k(\eta,\la^r)}\sum_{l=1\atop l\not=k}^n\d_{u_l}b_k(\eta,\la^r,0)
\d_tu_l(\eta,\tau_k(\eta,\xi,\tau,\lambda^r,\om^r))d\eta.
\ee
Inserting \reff{int2} into \reff{int1}, we get the integrals
\begin{eqnarray}
\lefteqn{
\int_{x_j}^x\int_{x_j}^\xi d_{jkl}^r(\xi,\eta,x)\d_tu_l(\eta,\tau_k(\eta,\xi,\tau_j(\xi,x,t,\lambda^r,\om^r),\lambda^r,\om^r))d\eta d\xi}
\nonumber\\
&&=\int_{x_j}^x\int_\eta^{x_j}
d_{jkl}^r(\xi,\eta,x)\d_tu_l(\eta,\tau_k(\eta,\xi,\tau_j(\xi,x,t,\lambda^r,\om^r),\lambda^r,\om^r))d\xi d\eta
\label{xieta}
\end{eqnarray}
with $j \not= k$ and  $l \not= k$ and
$$
d_{jkl}^r(\xi,\eta,x):=\frac{c_j(\xi,x,\la^r)c_k(\eta,\xi,\la^r)\d_{u_k}b_j(\xi,\la^r,0)\d_{u_l}b_k(\eta,\la^r,0)}
{a_j(\xi,\la^r)a_k(\eta,\la^r)}.
$$
Moreover, from \reff{charom} it follows
\begin{eqnarray}
\lefteqn{
\frac{d}{d\xi}u_l(\eta,\tau_k(\eta,\xi,\tau_j(\xi,x,t,\lambda^r,\om^r),\lambda^r,\om^r))}\nonumber\\
&&=\om^r\left(\frac{1}{a_j(\xi,\la^r)}-\frac{1}{a_k(\xi,\la^r)}\right)
\d_tu_l(\eta,\tau_k(\eta,\xi,\tau_j(\xi,x,t,\lambda^r,\om^r),\lambda^r,\om^r)).
\label{dxi}
\end{eqnarray}
By \reff{dxi} and  assumption \reff{hyp}, the right-hand side of \reff{xieta} equals 
$$
\frac{1}{\om^r}\int_0^x\int_\eta^x
\frac{a_j(\xi,\la^r)a_k(\xi,\la^r)}{a_k(\xi,\la^r)-a_j(\xi,\la^r)}
d_{jkl}^r(\xi,\eta,x)\frac{d}{d\xi}u_l(\eta,\tau_k(\eta,\xi,\tau_j(\xi,x,t,\lambda^r,\om^r),\lambda^r,\om^r))d\xi d\eta.
$$
Integrating by parts in the inner integral (with respect to $\xi$) we see that the absolute values of these integrals 
can be estimated by a constant times  $\|u\|_\infty$, where the constant does not depend on
$x$, $t$, $r$, and $u$.

Now, let us consider $H^r\d_uF(\la^r,\om^r,0)  \d_tu$. We have
\begin{eqnarray*}
\lefteqn{
(H^r\d_uF(\la^r,\om^r,0) \d_tu)_j(x,t)}\nonumber\\
&&=
-\frac{\om^r}{a_j(x,\la^r)}\int_{x_j}^x \frac{c_j(\xi,x,\la^r)}{a_j(\xi,\la^r)}\sum_{k=1\atop k\not=j}^n\d_{u_k}b_j(\xi,\la^r,0)
v_k^r(\xi,\tau_j(\xi,x,t,\lambda^r,\om^r)))d\xi
\end{eqnarray*}
with \reff{int2}. Proceeding as above one shows that these integrals can be estimated by a constant times  $\|u\|_\infty$, where the constant does not depend on
$x$, $t$, $r$, and $u$.

Now we give the proof of the claim for the operators $\d_uF(\la^r,\om^r,0)C(\la^r,\om^r)$.
Because of $\d_tC(\la^r,\om^r)u=C(\la^r,\om^r)\d_tu$ (cf. \reff{Abl}) and of 
\reff{du} we have for all $u \in \CC_n^1$ that
\begin{eqnarray*}
\d_x\d_uF(\la^r,\om^r,0)C(\la^r,\om^r)u&=&G^rC(\la^r,\om^r)u+H^rC(\la^r,\om^r)\d_tu,\\
\d_t\d_uF(\la^r,\om^r,0)C(\la^r,\om^r)u&=&\d_uF(\la^r,\om^r,0)C(\la^r,\om^r)\d_tu.
\end{eqnarray*}
It remains to show that 
$$
\|H^rC(\la^r,\om^r)\d_tu\|_\infty+\|\d_uF(\la^r,\om^r,0)C(\la^r,\om^r)\d_tu\|_\infty \le \mbox{const }\|u\|_\infty
\mbox{ for all } u \in \CC_n^1.
$$
Because of \reff{Cdef} and \reff{dF1} 
we have
$$
\left(\d_uF(\la^r,\om^r,0)C(\la^r,\om^r)\d_tu\right)_j(x,t)=
\sum_{k=1\atop k\not=j}^m\sum_{l=m+1}^nf^r_{jkl}(x,t)+
\sum_{k=m+1\atop k\not=j}^n\sum_{l=1}^mf^r_{jkl}(x,t)
$$
with
\begin{eqnarray*}
\lefteqn{
\displaystyle f^r_{jkl}(x,t)}\\
&\displaystyle :=
\int^{x_j}_x\frac{r_{kl}c_j(\xi,x,\la^r)c_k(x_k,\xi,\la^r)
\d_{u_k}b_k(\xi,\la^r,0)}{a_j(\xi,\la^r)}\d_tu_l(x_k,\tau_k(x_k,\xi,\tau_j(\xi,x,t,\la^r,\om^r),\lambda^r,\om^r))d\xi.&
\end{eqnarray*}
Using \reff{dxi} and assumption \reff{hyp}, one can integrate by parts in this  integral in order to see that 
the absolute values of these integrals 
can be estimated by a constant times $\|u\|_\infty$, where the constant does not depend on
$x$, $t$, $r$, and $u$. 
\end{subproof}

\begin{claim} \label{cl:6}
The set  $\left\{\left((I-C(\la^r,\om^r))^{-1}\d_uF(\la^r,\om^r,0)\right)^2w^r: \; r \in \N\right\}$
is precompact in $\CC_n$.
\end{claim}

\begin{subproof}
 By Claim \ref{cl:5} the sequence $\d_uF(\la^r,\om^r,0)^2w^r$ is bounded in $\CC_n^1$. 
Hence,  Lemma \ref{Fredho} (i) yields the same for the sequence
$
(I-C(\la^r,\om^r))^{-1}\d_uF(\la^r,\om^r,0)^2w^r.
$
Similarly, because of  Lemma \ref{Fredho} (i) and \reff{cont3}, 
the sequence $(I-C(\la^r,\om^r))^{-1}\d_uF(\la^r,\om^r,0)w^r$
is bounded in $\CC_n$.
Hence,  due to Claim \ref{cl:5}, we have same boundedness property for the sequence
$$
(I-C(\la^r,\om^r))^{-1}\d_uF(\la^r,\om^r,0)C(\la^r,\om^r)(I-C(\la^r,\om^r))^{-1}
\d_uF(\la^r,\om^r,0)w^r. 
$$
Therefore, the sequence
\begin{eqnarray*}
\lefteqn{
(I-C(\la^r,\om^r))^{-1}\d_uF(\la^r,\om^r,0)(I-C(\la^r,\om^r))^{-1}\d_uF(\la^r,\om^r,0)w^r}\\
&&=(I-C(\la^r,\om^r))^{-1}\d_uF(\la^r,\om^r,0)^2w^r\\
&&+(I-C(\la^r,\om^r))^{-1}\d_uF(\la^r,\om^r,0)C(\la^r,\om^r)
(I-C(\la^r,\om^r))^{-1}\d_uF(\la^r,\om^r,0)w^r 
\end{eqnarray*}
is bounded in $\CC_n^1$. Hence, the Arcela-Ascoli theorem yields  the claim.
\end{subproof}

\begin{claim} \label{cl:7}
A subsequence of $\left\{w^r: \; r \in \N\right\}$ converges to zero in $\CC_n$.
\end{claim}

\begin{subproof}
Denote $g^r:=(I-C(\la^r,\om^r)-\d_uF(\la^r,\om^r,0))w^r$. 
By Claim \ref{cl:1},  Lemma \ref{Fredho}~(i), and \reff{cont3}, we have 
\begin{eqnarray}
\lefteqn{
\left(I+(I-C(\la^r,\om^r))^{-1}\d_uF(\la^r,\om^r,0)\right)(I-C(\la^r,\om^r))^{-1}
(I-P)g^r}
\nonumber\\
&&
=\left(I-\left[(I-C(\la^r,\om^r))^{-1}\d_uF(\la^r,\om^r,0)\right]^2\right)w^r\nonumber\\
&&
-\left(I+(I-C(\la^r,\om^r))^{-1}\d_uF(\la^r,\om^r,0)\right)(I-C(\la^r,\om^r))^{-1}Pg^r\to 0
\mbox{ in } \CC_n.\label{P}
\end{eqnarray}
Moreover, by Claim \ref{cl:6}, a subsequence of 
$\left(I-\left[(I-C(\la^r,\om^r))^{-1}\d_uF(\la^r,\om^r,0)\right]^2\right)w^r$
converges in $\CC_n$. Further, due to \reff{cont3}, the sequence 
$g^r$ is bounded  in $\CC_n$. Then, taking into account  Lemma \ref{Fredho} (i), \reff{cont3},
and the fact that
$\dim P<\infty$, we conclude that
a subsequence of 
$\left(I+(I-C(\la^r,\om^r))^{-1}\d_uF(\la^r,\om^r,0)\right)(I-C(\la^r,\om^r))^{-1}Pg^r$
converges  in $\CC_n$. Now we get from \reff{P} that a subsequence of $w^r$
(which will be denoted by $w^r$ again) converges  in $\CC_n$, i.e., 
\beq\label{conv_w}
w^r  \to  w^* \mbox{ for some } w^* \in \ker Q. 
\ee
On the other hand, Claims \ref{cl:1} and \ref{cl:2} imply that
\beq
\label{I-Pweak}
(I-P)\left(I-C(0,1)-\d_uF(0,1,0)\right)w^r \rightharpoonup 0  \mbox{ in } 
L^2\left((0,1)\times(0,2\pi);\R^n\right).
\ee
Convergences \reff{conv_w} and \reff{I-Pweak} entail that
$$
(I-P)\left(I-C(0,1)-\d_uF(0,1,0)\right)w^*=0,
$$
i.e.,   
$w^* \in \ker Q \cap \im Q$, hence, $w^*=0$ as desired.
\end{subproof}

The proof of Lemma \ref{Magnus} is therewith complete.
\end{proof}

Now we are well-prepared for solving \reff{infinite} locally with respect to $w$ by an 
implicit-function-theorem-type argument.

For $\delta>0$ we denote
$$
\B_\delta:=\{v \in \CC_n: \; \|v\|_\infty \le \delta\}.
$$

\begin{lemma}\label{infIFT}
(i) There exists $\delta_3 \in (0,\delta_2)$ such that for all $\la \in
[-\delta_3,\delta_3]$, $\om \in [1-\delta_3,1+\delta_3]$, and $v \in \B_{\delta_3} \cap\im Q$
there exists exactly one solution 
$w=\hat{w}(\la,\om,v)$ to \reff{infinite} with $w \in \B_{\delta_3} \cap \ker Q$. 

(ii) There exists $c_3>0$ such that for all $\la \in [-\delta_3,\delta_3]$, $\om \in [1-\delta_3,1+\delta_3]$,
and  $v \in  \B_{\delta_3}\cap \im Q$ it holds
$$
\|\hat{w}(\la,\om,v)\|_\infty \le c_3 \|v\|^2_\infty.
$$

(iii) The map $(\la,\om,v)  \in [-\delta_3,\delta_3] \times  [1-\delta_3,1+\delta_3]
\times  (\B_{\delta_3}\cap \im Q) \mapsto \hat{w}(\la,\om,v) \in  \CC_n$ is continuous.

(iv) For all $\la \in [-\delta_3,\delta_3]$ and  $\om \in [1-\delta_3,1+\delta_3]$
the maps $v \in  \B_{\delta_3}\cap\im Q \mapsto \hat{w}(\la,\om,v) \in \CC_n$ are   $C^\infty$-smooth.
\end{lemma}
\begin{proof} Consider the map $\F:[-\delta_0,\delta_0]\times \R \times \im Q \times \ker Q \to \ker P$ defined by
\beq
\label{FFdef}
\F(\la,\om,v,w):=(I-P)((I-C(\la,\om))(v+w)-F(\la,\om,v+w)).
\ee
Obviously, the map $\F$ has  properties analogous to \reff{cont1}--\reff{cont3}, in particular, it is continuous, 
\beq
\label{Fsmooth}
\F(\la,\om,\cdot,\cdot) \mbox{ is $C^\infty$-smooth}
\ee
and
\beq
\label{strcont}
(\la,\om) \in  [-\delta_3,\delta_3] \times \R \mapsto \d_w\F(\la,\om,v,w)w_1 \in \CC_n \mbox{ is continuous } 
\ee
for all $v \in \im Q$ and  $w,w_1 \in  \ker Q$.
Because of Lemma \ref{Fredho} (ii) and Lemma \ref{Magnus}, for all  $\la \in [-\delta_2,\delta_2]$ and $\mu \in [1-\delta_2,1+\delta_2]$
the operator
$$
\partial_w\F(\la,\om,0,0)=(I-P)(I-C(\la,\om)-\d_uF(\la,\om,0))
$$
is an injective Fredholm operator from $\ker Q$ into $\ker P$. Its index is zero because the index of $I-P$ (as an operator from $\CC_n$ into $\ker P$)
is $-2$, the index of $I-C(\la,\om)-\d_uF(\la,\om,0)$ (as an operator from $\CC_n$ into $\CC_n$) is zero, and the index of the embedding $w \in \ker Q 
\mapsto w \in \CC_n$ is $2$. Hence, 
$\partial_w\F(\la,\om,0,0)$ is an isomorphism 
from $\ker Q$ onto $\ker P$, and Lemma \ref{Magnus} yields that for all $\la \in [-\delta_2,\delta_2]$ and $\om \in [1-\delta_2,1+\delta_2]$ it holds
\beq
\label{c_1}
\|\partial_w\F(\la,\om,0,0)^{-1}\|_{\LL(\ker P;\ker Q)} \le c_2.
\ee
Furthermore,
\beq
\label{strcont1}
(\la,\om) \in  [-\delta_3,\delta_3] \times \R \mapsto \d_w\F(\la,\om,0,0)^{-1}w \in \CC_n \mbox{ is continuous for all } w \in  \ker Q.
\ee
Indeed, on the account of  \reff{strcont} and \reff{c_1}, we get that
\begin{eqnarray*}
&&\left(\d_w\F(\la,\om,0,0)^{-1}- \d_w\F(\la^0,\om^0,0,0)^{-1}\right)w\\
&&=\d_w\F(\la,\om,0,0)^{-1}
\left(\d_w\F(\la^0,\om^0,0,0)-\d_w\F(\la,\om,0,0)\right)\d_w\F(\la^0,\om^0,0,0)^{-1}w
\end{eqnarray*}
tends to zero for $(\la,\om) \to (\la^0,\om^0)$.

We have to solve \reff{infinite}, i.e.,  $\F(\la,\om,v,w)=0$.  For   $\la \in [-\delta_2,\delta_2]$ and $\om \in [1-\delta_2,1+\delta_2]$
this is equivalent 
to  the fixed point problem 
\beq
\label{FP}
\G(\la,\om,v,w):=w-\d_w\F(\la,\om,0,0)^{-1}\F(\la,\om,v,w)=w.
\ee
Moreover, it holds
\begin{eqnarray*}
\lefteqn{
\G(\la,\om,v,w^1)-\G(\la,\om,v,w^2)=\int_0^1\d_w\G(\la,\om,v,sw^1+(1-s)w^2)(w_1-w_2)ds}\nonumber\\
&&=\d_w\F(\la,\om,0,0)^{-1}\int_0^1\left(\d_w\F(\la,\om,0,0)-
\d_w\F(\la,\om,v,sw^1+(1-s)w^2)\right)(w^1-w^2)ds,
\end{eqnarray*}
where
\begin{eqnarray*}
\lefteqn{
\left(\d_w\F(\la,\om,v,sw^1+(1-s)w^2)-\d_w\F(\la,\om,0,0)\right)(w^1-w^2)}\\
&&=(I-P)\left(\d_uF(\la,\om,v+sw^1+(1-s)w^2)-\d_uF(\la,\om,0)\right)(w^1-w^2)\\
&&=(I-P)\int_0^1\d_u^2F(\la,\om,r(v+sw^1+(1-s)w^2))(v+sw^1+(1-s)w^2,w^1-w^2)dr.
\end{eqnarray*}
Hence, \reff{cont3} and \reff{c_1}
yield that there exists $\delta_3 \in (0,\delta_2)$ such that for all $\la \in
[-\delta_3,\delta_3]$, $\om \in [1-\delta_3,1+\delta_3]$, and $v \in \B_{\delta_3} \cap\im Q$
we have
$$
\|\G(\la,\om,v,w^1)-\G(\la,\om,v,w^2)\|_\infty\le \frac{1}{2}\|w^1-w^2\|_\infty \mbox{ for all }
w^1,w^2 \in \B_{\delta_3}\cap\ker Q.
$$
In other words: For those $\la$, $\om$, and $v$ the map $\G(\la,\om,v,\cdot)$ is strictly contractive on  $\B_{\delta_3}\cap\ker Q$.
In order to apply Banach's fixed point theorem we have to show that 
$\G(\la,\om,v,\cdot)$ maps  $\B_{\delta_3}\cap\ker Q$ into itself if $\delta_3$
is chosen sufficiently small.

Using \reff{c_1} again, for all  $\la \in
[-\delta_3,\delta_3]$, $\om \in [1-\delta_3,1+\delta_3]$, and $v \in \B_{\delta_3}\cap\im Q$
we get
\begin{eqnarray}
\|\G(\la,\om,v,w)\|_\infty &\le& \|\G(\la,\om,v,w)-\G(\la,\om,v,0)\|_\infty 
+\|\G(\la,\om,v,0)\|_\infty\nonumber\\
\label{absch1}
&\le&
\frac{1}{2}\|w\|_\infty+c_2\|\F(\la,\om,v,0)\|_\infty \mbox{ for all }
w \in \B_{\delta_3}\cap \ker Q.
\end{eqnarray}
Moreover, because of $v \in \ker L_0$ it holds
\begin{eqnarray*}
\lefteqn{
\F(\la,\om,v,0)=(I-P)((I-C(\la,\om))v-F(\la,\om,v))}\nonumber\\
&&=(I-P)(\d_uF(\la,\om,0)v-F(\la,\om,v))
=(I-P)\int_0^1s\partial^2_uF(\la,\om,sv)(v,v)ds.
\end{eqnarray*}
Hence, \reff{cont3} yields that there exists $c_3>0$ such that for all   $\la \in
[-\delta_3,\delta_3]$, $\om \in [1-\delta_3,1+\delta_3]$, and $v \in \B_{\delta_3}\cap\im Q$ we have 
\beq
\label{vAbsch}
\|\F(\la,\om,v,0)\|_\infty \le \frac{c_3}{2c_2}\|v\|^2_\infty,
\ee
and, because of \reff{absch1}, it follows that $\G(\la,\om,v,w) \in  \B_{\delta_3}\cap\im Q$ for all $ w \in  \B_{\delta_3}\cap\im Q$
if $\delta_3\le 1/c_3$.

Now, Banach's fixed point theorem gives a unique
in  $\B_{\delta_3}\cap\ker Q$  solution $w=\hat{w}(\la,\om,v)$ to \reff{FP} for all 
 $\la \in
[-\delta_3,\delta_3]$, $\om \in [1-\delta_3,1+\delta_3]$, and $v \in \B_{\delta_3}\cap \im Q$.
Hence, assertion (i) of the lemma is proved.

Moreover, \reff{absch1} and \reff{vAbsch} yield assertion (ii).

Assertion (iii) follows from the uniform contraction principle (cf. \cite[Theorem 1.244]{Chicone}):
First note that the contraction constant of $\G(\la,\om,v,\cdot)$ does not depend on $(\la,\om,v)$.
Moreover, for all $\lambda,\la^0 \in [\delta_2,\delta_2]$, $\om,\om^0 \in [1-\delta_2,1+\delta_2]$,
$v,v^0 \in \B_{\delta} \cap \im Q$, and $w,w^0 \in \B_{\delta} \cap \ker Q$ the difference
\begin{eqnarray*}
&&\G(\la,\om,v,w)-\G(\la^0,\om^0,v^0,w^0)\\
&&=w-w^0-\d_w\F(\la,\om,0,0)^{-1}\left(\F(\la,\om,v,w)-\F(\la^0,\om^0,v^0,w^0)\right)\\
&&+\left(\d_w\F(\la^0,\om^0,0,0)^{-1}-\d_w\F(\la,\om,0,0)^{-1}\right)\F(\la^0,\om^0,v^0,w^0),
\end{eqnarray*}
tends to zero in $\CC_n$ for $(\la,\om,v,w) \to (\la^0,\om^0,v^0,w^0)$
because of the continuity of $\F$,  \reff{c_1}, and   \reff{strcont1}. Hence, $\G$ is continuous.

Finally, assertion (iv) 
follows from the classical implicit function theorem.
\end{proof}

\subsection{Smoothness with respect to $x$ and $t$ }
\label{Additional Regularity}

The aim of this subsection is to prove the following result:
\begin{lemma}
\label{regular}
For all  $\la \in
[-\delta_3,\delta_3]$, $\om \in [1-\delta_3,1+\delta_3]$, and $v \in  \B_{\delta_3}\cap\im Q$
the map 
$$
(x,t) \in [0,1] \times \R \mapsto
[\hat{w}(\la,\om,v)](x,t) \in \R^n
$$
is $C^\infty$-smooth.
\end{lemma}
\begin{proof}
For $\vphi \in \R$ and $u \in \CC_n$ we define $S_\vphi u \in \CC_n$ by
$
[S_\vphi u](x,t):=u(x,t+\vphi). 
$
Take  $\la \in
[-\delta_3,\delta_3]$, $\om \in [1-\delta_3,1+\delta_3]$, and $v \in  \B_{\delta_3}\cap \im Q$.
It is easy to see that for all $u \in \CC_n$ it holds
$$
S_\vphi C(\la,\om)u= C(\la,\om)S_\vphi u, \; 
S_\vphi F(\la,\om,u)=F(\la,\om,S_\vphi u).
$$
Moreover, the definitions \reff{def}, \reff{vtilde}, and \reff{Pdef} imply that for all $u \in \CC_n$ it holds
\begin{eqnarray}
PS_\vphi u&=&\sum_{k=1}^2\langle S_\vphi u,\tilde{A}{\bf w_k}\rangle {\bf \tilde{v}_k}=\sum_{k=1}^2\langle u,\tilde{A} S_{-\vphi}{\bf w_k}\rangle 
{\bf \tilde{v}_k}\nonumber\\
&=&\langle u,\tilde{A}(\cos \vphi {\bf w_1}-\sin \vphi {\bf w_2})\rangle {\bf \tilde{v}_1}+\langle u,\tilde{A}(\cos \vphi {\bf w_2}
+\sin \vphi {\bf w_1})\rangle {\bf \tilde{v}_2}\nonumber\\
&=&\langle u,\tilde{A}{\bf w_1}\rangle (\cos \vphi {\bf \tilde{v}_1}+\sin \vphi {\bf \tilde{v}_2})+
\langle u,\tilde{A}{\bf w_2}\rangle (\cos \vphi {\bf \tilde{v}_2}-\sin \vphi {\bf \tilde{v}_1})\nonumber\\
&=&S_\vphi Pu.
\label{Pinvert}
\end{eqnarray}
Similarly one shows $Q S_\vphi u=S_\vphi Qu$.
Hence, applying $S_\vphi$ to the identity
\beq
\label{41}
(I-P)((I-C(\la,\om))(v+\hat{w}(\la,\om,v))-F(\la,\om,v+\hat{w}(\la,\om,v)))=0,
\ee 
we get $(I-P)((I-C(\la,\om))(S_\vphi v+S_\vphi \hat{w}(\la,\om,v))-F(\la,\om,S_\vphi v+S_\vphi \hat{w}(\la,\om,v)))=0$.
Therefore, the uniqueness assertion of Lemma \ref{infIFT} yields
\beq
\label{42}
S_\vphi \hat{w}(\la,\om,v)=\hat{w}(\la,\om,S_\vphi v).
\ee
But all functions $v \in \ker L_0= \im Q$ are $C^\infty$-smooth, hence the maps $\vphi \in \R \mapsto S_\vphi v \in \CC_n$ are 
$C^\infty$-smooth for all  $v \in \im Q$. By Lemma \ref{infIFT} (iv), 
$\vphi \in \R \mapsto  S_\vphi \hat{w}(\la,\om,v) \in \CC_n$ is $C^\infty$-smooth, i.e., 
\beq
\label{glatt}
t \in \R \mapsto  \hat{w}(\la,\om,v)(x,t) \in \R^n \mbox{ is $C^\infty$-smooth}.
\ee

On the other hand, \reff{41} can be rewritten as follows:
\begin{eqnarray}
\lefteqn{
\!\!\!\!\!\!(I-C(\la,\om))\hat{w}(\la,\om,v)-F(\la,\om,v+\hat{w}(\la,\om,v))}\nonumber\\
&&\!\!\!\!\!\!\!\!\!\!\!\!
=(C(\la,\om)-I)v +P\left((I-C(\la,\om))(v+\hat{w}(\la,\om,v))-F(\la,\om,v+\hat{w}(\la,\om,v))\right).
\label{44}
\end{eqnarray}
The right-hand side of \reff{44} is $C^\infty$-smooth with respect to $(x,t)$. 
Moreover, from \reff{Cdef}, \reff{Fdef}, and \reff{glatt} 
it follows that $C(\la,\om)\hat{w}(\la,\om,v)$
and $F(\la,\om,v+\hat{w}(\la,\om,v))$ are  $C^\infty$-smooth with respect to $(x,t)$. 
Hence, \reff{44} yields that $\hat{w}(\la,\om,v)$  is  $C^\infty$-smooth with respect to~$(x,t)$. 
\end{proof}

\subsection{Differentiability with respect to $\la$ and $\om$ }
\label{smooth dep}

In this subsection we show that the map $\hat{w}$ is $C^2$-smooth.
For that we  use the well-known fact (see, e.g. \cite[Section 1.11.3]{Chicone})
that the fiber contraction principle can be used to show 
that functions, which are fixed points of certain operators in $C$-spaces, are smooth.

Remark that in the particular case, when all coefficient functions 
$a_j(x,\la)$ are $\la$-inde\-pen\-dent (and, hence, the problem   
(\ref{eq:1.1})--(\ref{eq:1.3}) depends
on the bifurcation parameter $\la$ via the terms $b_{jk}(x,\la,u)$ only),
the maps $\la \mapsto \hat{w}(\la,\om,v)$ are $C^\infty$-smooth because in this case the 
characteristics $\tau_j(\xi,x,t,\la,\om)$
do not depend on $\la$ and, hence, the maps $\la \mapsto C(\la,\om)$ and $\la \mapsto F(\la,\om,u)$ 
are  $C^\infty$-smooth (cf. \reff{Cdef} and \reff{Fdef}).
But also in this case the question if the maps $\om \mapsto \hat{w}(\la,\om,v)$ are differentiable, 
remains to be difficult.

\begin{lemma}
\label{laregular}
The map
$$
(\la,\om,v)  \in
[-\delta_3,\delta_3] \times [1-\delta_3,1+\delta_3] \times (\B_{\delta_3}\cap\im Q)
\mapsto
\hat{w}(\la,\om,v) \in \CC_n
$$ 
is $C^2$-smooth. 
\end{lemma}
\begin{proof}
For $k=0,1,2,\ldots$ we define 
maps $w_k:[-\delta_3,\delta_3] \times [1-\delta_3,1+\delta_3] \times (\B_{\delta_3}\cap\im Q)
\to \ker Q$ 
by means of
\beq
\label{sequence}
w_0:=0,\; w_{k+1}(\la,\om,v):=
w_k(\la,\om,v)-\d_w\F(\la,\om,0,0)^{-1}\F(\la,\om,v,w_k(\la,\om,v)).
\ee
As it follows from  the proof of Lemma \ref{infIFT},
\beq
\label{wkconv}
w_k(\la,\om,v) \to \hat{w}(\la,\om,v) \mbox{ in } \CC_n \mbox{ for } k \to \infty \mbox{ uniformly with respect to } \la,\om,v.
\ee
Now we are going to show that all functions $w_1,w_2,\ldots$ are $C^2$-smooth 
and that the sequences of all their partial derivatives up to the second order converge in $\CC_n$
uniformly with respect to $\la, \om$, and $v$.
Then the classical theorem of calculus  \cite[Theorem 8.6.3]{Dieudonne}
yields that  $\hat{w}$ is $C^2$-smooth (and, hence, the lemma is proved). 
Here and in what follows we identify the two-dimensional vector space $\im Q$ with $\R^2$ (by fixing a certain basis in  $\im Q$), and the 
partial derivatives are taken with respect to the corresponding coordinates.

 For $l=1,2,\ldots$ we denote
$$
\tilde{\CC}_n^l:=\{u \in \CC_n:\,\d_t^ju \in \CC_n \mbox{ for } j=1,2,\ldots,l\}.
$$
This is a Banach space with the norm
$$
\|u\|_l:=\sum_{j=0}^l\|\d_t^ju\|_\infty.
$$

The remainder of the proof will be divided into a number of claims.

\begin{cclaim}\label{cl:11}
 For all 
$\la \in[-\delta_3,\delta_3]$,  $\om \in [1-\delta_3,1+\delta_3]$, $v \in  \B_{\delta_3}\cap\im Q$, and $k,l=1,2,\ldots$
the function $w_k(\la,\om,v)$ belongs to  $\tilde{\CC}_n^l$. Furthermore, 
 $w_k(\la,\om,v)$ depends continuously
(with respect to the norm in $\tilde{\CC}_n^l$) on $\la$, $\om$, and $v$. 
\end{cclaim}

\begin{subproof}
The proof will be done by induction on $k$.
For $k=0$ this claim is obvious.

To do the induction step we  proceed as in the proof of Lemma \ref{regular}.
For all $\vphi \in \R$ we have
\beq
\label{invar}
S_\vphi w_{k+1}(\la,\om,v)=S_\vphi w_{k}(\la,\om,v)-\d_w\F(\la,\om,0,0)^{-1}\F(\la,\om,v,S_\vphi w_k(\la,\om,v)).
\ee
By induction assumption the map $\vphi \in \R \mapsto S_\vphi w_{k}(\la,\om,v) \in \CC_n$ is $C^\infty$-smooth, and all its derivatives depend continuously 
(with respect to $\|\cdot\|_\infty$)  on $\la$, $\om$, and $v$.
Hence, by \reff{Fsmooth} and  \reff{strcont1}, the right-hand side of \reff{invar} is $C^\infty$-smooth with respect to $\vphi$, and all its derivatives  in $\vphi$
depend continuously 
(with respect to $\|\cdot\|_\infty$)  on $\la$, $\om$, and $v$.
Therefore, also the left-hand side  of \reff{invar} has this property, i.e., all partial derivatives $\d_t^jw_{k+1}(\la,\om,v)$ exist in $\CC_n$ 
and depend continuously 
(with respect to $\|\cdot\|_\infty$)  on $\la$, $\om$, and $v$.
\end{subproof}

\begin{cclaim}\label{cl:12}
The sequence $\d_tw_k(\la,\om,v)$ converges in $\CC_n$ for $k\to\infty$ uniformly with respect to 
$\la, \om$, and $v$.
\end{cclaim}

\begin{subproof}
From \reff{Abl} and \reff{Pinvert} (which implies that $\d_tPu=P\d_tu$ for all 
$u \in \tilde{\CC}_n^1$)
it follows
that for all $w \in  \tilde{\CC}_n^1 \cap \ker Q$ we have
\beq
\label{tF}
\d_t\F(\la,\om,v,w)=\d_w\F(\la,\om,v,w)\d_tw+\d_v\F(\la,\om,v,w)\d_tv.
\ee
Hence,
\begin{eqnarray*}
\lefteqn{
\d_tw_{k+1}(\la,\om,v)}\\
&&=\left(I-\d_w\F(\la,\om,0,0)^{-1}\d_w\F(\la,\om,v,w_k(\la,\om,v))\right)\d_tw_k(\la,\om,v)\\
&&-\d_w\F(\la,\om,0,0)^{-1}\d_v\F(\la,\om,v,w_k(\la,\om,v))\d_tv.
\end{eqnarray*}
But \reff{wkconv} implies
$$
\d_w\F(\la,\om,0,0)^{-1}\d_v\F(\la,\om,v,w_k(\la,\om,v))\d_tv \to \d_w\F(\la,\om,0,0)^{-1}\d_v\F(\la,\om,v,\hat{w}(\la,\om,v))\d_tv 
$$
in  $\CC_n$ for  $k \to \infty$ uniformly with respect to  $\la,\om$, and $v$. Moreover, from the proof of
 Lemma \ref{infIFT} it follows that
$$
\|I-\d_w\F(\la,\om,0,0)^{-1}\d_w\F(\la,\om,v,w_k(\la,\om,v))\|_{\LL(\ker Q)} \le \frac{2}{3}
$$
for large $k$. Hence, Lemma~\ref{A2}
yields the desired claim. 
Here we apply Lemma~\ref{A2} with $U$ chosen as the Banach space of all continuous maps 
$u:[-\delta_3,\delta_3] \times [1-\delta_3,1+\delta_3] 
\times (\B_{\delta_3}\cap\im Q) \to \ker Q$ with the usual maximum norm.
\end{subproof}

\begin{cclaim}\label{cl:13}
For $j=1,2$ the partial derivative  $\d_{v_j}w_k(\la,\om,v)$  exists in $\CC_n$
and depends continuously
(with respect to the norm  $\|\cdot\|_\infty$ in $\CC_n$) on $\la$, $\om$, and $v$.
\end{cclaim}

\begin{subproof}
The proof is by induction on $k$.  

Due to induction assumption,   $\d_{v_j}w_{k}(\la,\om,v)$ exists in $\CC_n$. Hence, 
on the account of  \reff{Fsmooth},
the partial derivative with respect to $v_j$ of 
$$
w_{k+1}(\la,\om,\cdot)=w_{k}(\la,\om,\cdot)- \d_w\F(\la,\om,0,0)^{-1}\F(\la,\om,\cdot,w_k(\la,\om,\cdot))
$$
exists in $\CC_n$. 
\end{subproof}

Similarly one proves the following statement.

\begin{cclaim}\label{cl:130}
All partial derivatives $\d_t^l\d_{v_j}^mw_k(\la,\om,v)$ exist in $\CC_n$
and depend continuously
(with respect to the norm in $\CC_n$) on $\la$, $\om$, and $v$.
\end{cclaim}

\begin{cclaim}\label{cl:14}
Given $j=1,2$, the sequence $\d_{v_j}w_k(\la,\om,v)$ converges  in $\CC_n$
for $k \to \infty$  uniformly with respect to  $\la,\om$, and $v$.
\end{cclaim}

\begin{subproof}
 This follows, as above, from the equality
\begin{eqnarray*}
\lefteqn{
\d_{v_j}w_{k+1}(\la,\om,v)}\\
&&=\left(I-\d_w\F(\la,\om,0,0)^{-1}\d_w\F(\la,\om,v,w_k(\la,\om,v))\right)\d_{v_j}w_k(\la,\om,v)\\
&&-\d_w\F(\la,\om,0,0)^{-1}\d_{v_j}\F(\la,\om,v,w_k(\la,\om,v))
\end{eqnarray*}
and Lemma \ref{A2}.
\end{subproof}

\begin{cclaim}\label{cl:15}
For all $k=1,2,\dots$ the partial 
derivative $\partial_\la w_k(\la,\om,v)$ exists, belongs to $\CC_n$,
and depends continuously
(with respect to $\|\cdot\|_\infty$) on $\la$, $\om$, and $v$.
\end{cclaim}

\begin{subproof}
For the proof we use Claim \ref{cl:11} and induction on $k$: 
For $k=0$ the claim is obvious.

In order to do the induction step, first  let us do some preliminary calculations.
It follows directly from the definitions \reff{Cdef} and \reff{Fdef} that for all $\om \in \R$ and $u \in \tilde{\CC}_n^1$ 
the maps $\la \in [-\delta_0,\delta_0] \mapsto C(\la,\om)u \in \CC_n$ and $\la \in [-\delta_0,\delta_0] \mapsto F(\la,\om,u) \in \CC_n$ 
are differentiable and
\beq
\label{lader}
\begin{array}{rcl}
\d_\la C(\la,\om)u&=&C^\la_0(\la,\om)u+C^\la_1(\la,\om)\d_tu,\\
\d_\la F(\la,\om,u)&=&F^\la_0(\la,\om,u)+F^\la_1(\la,\om,u)\d_tu,
\end{array}
\ee
where the operators $C^\la_0(\la,\om),C^\la_1(\la,\om) \in \LL(\CC_n)$ are defined as follows: For $j=1,\ldots,m$ we set
$$
\begin{array}{rcl}
\displaystyle
\left(C^\la_0(\la,\om)u\right)_j(x,t)&:=&
\displaystyle
\d_\la c_j(0,x,\la)\sum_{k=m+1}^nr_{jk}u_k(0,\tau_j(0,x,t,\la,\om)),\\
\displaystyle
\left(C^\la_1(\la,\om)v\right)_j(x,t)&:=&
\displaystyle
\d_\la\tau_j(0,x,t,\la,\om)c_j(0,x,\la)\sum_{k=m+1}^nr_{jk}v_k(0,\tau_j(0,x,t,\la,\om)),
\end{array}
$$
and for $j=m+1,\ldots,n$ we set
$$
\begin{array}{rcl}
\displaystyle
\left(C^\la_0(\la,\om)u\right)_j(x,t)&:=&
\displaystyle
\d_\la c_j(1,x,\la)\sum_{k=1}^mr_{jk}u_k(1,\tau_j(1,x,t,\la,\om)),\\
\displaystyle
\left(C^\la_1(\la,\om)v\right)_j(x,t)&:=&
\displaystyle
\d_\la\tau_j(1,x,t,\la,\om)c_j(1,x,\la)\sum_{k=1}^mr_{jk}v_k(1,\tau_j(1,x,t,\la,\om)).
\end{array}
$$
Further, the map $F^\la_0: [-\delta_0,\delta_0] \times \R \times  \CC_n \to  \CC_n$ is defined by  
\begin{eqnarray*}
(F^\la_0(\la,\om,u))_j(x,t)&:=&
\int_{x_j}^x
\d_\la\left(\frac{c_j(\xi,x,\la)}{a_j(\xi,\la)}\right)
f_j(\xi,\la,u(\xi,\tau_j(\xi,x,t,\la,\om)))d\xi\\
&&+\int_{x_j}^x \frac{c_j(\xi,x,\la)}{a_j(\xi,\la)}\d_\la f_j(\xi,\la,u(\xi,\tau_j(\xi,x,t,\la,\om)))d\xi
\end{eqnarray*}
and $F^x_1(\la,\om,u) \in  \LL(\CC_n)$ is defined by
\begin{eqnarray*}
&&(F^x_1(\la,\om,u)v)_j(x,t):=\\
&&
\displaystyle\int_{x_j}^x\frac{\d_\la\tau_j(\xi,x,t,\la,\om)c_j(\xi,x,\la)}{a_j(\xi,\la)}\sum_{k=1}^n
\d_{u_k}f_j(\xi,\la,u(\xi,\tau_j(\xi,x,t,\la,\om)))
v_k(\xi,\tau_j(\xi,x,t,\la,\om))d\xi.
\end{eqnarray*}
Hence, for all $\om \in \R$, $v \in \mbox{im}\,Q$, and $w \in \tilde{\CC}_n^1 \cap \ker Q$ 
the map $\la \in [-\delta_0,\delta_0] \mapsto \F(\la,\om,v,w) \in \CC_n$ 
is differentiable and its derivative is
\begin{eqnarray}
\d_\la \F(\la,\om,v,w)&=&-(I-P)\left(C_0^\la(\la,\om)(v+w)+F_0^\la(\la,\om,v+w)\right.\nonumber\\
&&\left.+(C_1^\la(\la,\om)+F_1^\la(\la,\om,v+w))(\d_tv+\d_tw)\right).
\label{dla}
\end{eqnarray}
Remark that  the map $\F:[-\delta_0,\delta_0] \times \R \times  \mbox{im}\,Q \times \ker Q \to \ker P$ defined in \reff{FFdef}
is not differentiable (with respect to the norm $\|\cdot\|_\infty$ in $\ker P$ and $\ker Q$), in general.
But its restriction to $[-\delta_0,\delta_0] \times \R \times  \mbox{im}\,Q \times (\tilde{\CC}_n^1 \cap \ker Q)$
is differentiable  (with respect to the norms $\|\cdot\|_\infty$ in $\ker P$ and $\|\cdot\|_1$ in $\ker Q$).
Therefore, the notation $\d_\la \F$ should be used with care: It is not the partial derivative of $\F$, but of its 
restriction to $[-\delta_0,\delta_0] \times \R \times  \mbox{im}\,Q \times (\tilde{\CC}_n^1 \cap \ker Q)$.
In particular, it is a map from $[-\delta_0,\delta_0] \times \R \times \mbox{im}\,Q \times  (\tilde{\CC}_n^1 \cap \ker Q)$ into  $\CC_n$.
The same care is needed if using the notations $\d_\la C(\la,\om)u$ and $\d_\la F(\la,\om,u)$.

Now let us  do the induction step.  By the induction assumption, for a fixed $k$ the partial 
derivative $\partial_\la w_k(\la,\om,v)$ exists in $\CC_n$
and depends continuously on $\la$, $\om$, and $v$.
Define
$$
R_k(\la,\om,v):=\d_w\F(\la,\om,0,0)w_{k}(\la,\om,v)-\F(\la,\om,v,w_k(\la,\om,v)).
$$
Because of  $w_k(\la,\om,v) \in \tilde{\CC}_n^1$ the partial derivative $\d_\la R_k(\la,\om,v)$  exists in $\CC_n$
and depends continuously on $\la$, $\om$, and $v$. In fact, it holds
\begin{eqnarray*}
\lefteqn{
\d_\la R_k(\la,\om,v)=\d_\la\d_w\F(\la,\om,0,0) w_k(\la,\om,v)+\d_w\F(\la,\om,0,0) \d_\la w_k(\la,\om,v)}\\
&&-\d_\la \F(\la,\om,v,w_k(\la,\om,v))-\d_w \F(\la,\om,v,w_k(\la,\om,v)) \d_\la w_k(\la,\om,v).
\end{eqnarray*}
Here $\d_\la\d_w\F(\la,\om,0,0) \in \LL(\tilde{\CC}_n^1 \cap \ker Q;\CC_n)$ is the derivative of the map $\la \in [-\delta,\delta_0]
\mapsto \d_w\F(\la,\om,0,0) \in \LL(\tilde{\CC}_n^1 \cap \ker Q;\CC_n)$ or, the same, the derivative of the map $w \in \tilde{\CC}_n^1 \cap \ker Q
\mapsto \d_\la \F(\la,\om,0,w) \in \CC_n$ at $w=0$, i.e. (cf. \reff{dla}),
$$
\d_\la \d_w\F(\la,\om,0,0)w=(P-I)\left((C_0^\la(\la,\om)+\d_uF_0^\la(\la,\om,0))w+(C_1^\la(\la,\om)+F_1^\la(\la,\om,0))\d_tw\right).
$$
It is easy to check (using the definitions of $C_0^\la,C_1^\la, F_0^\la$, and $F_1^\la$) that for all $w \in  \tilde{\CC}_n^1 \cap\ker Q$ it holds
\beq
\label{strcont2}
(\la,\om) \in  [-\delta_3,\delta_3] \times \R \mapsto \d_\la\d_w\F(\la,\om,0,0)w \in \CC_n \mbox{ is continuous, } 
\ee
and that for all  $\la \in [-\delta_0,\delta_0]$,
$\om \in  \R$, and $w \in \tilde{\CC}_n^1 \cap \ker Q$ it holds
\beq
\label{lawest}
\|\d_\la \d_w\F(\la,\om,0,0)w\|_\infty \le \mbox{const }\|w\|_1,
\ee
where the constant does not depend on  $w$, $\la$, and $\om$ (as long as $\om$ varies in a bounded interval).

We have 
\beq
\label{Feq2}
\d_w\F(\la,\om,0,0)w_{k+1}(\la,\om,v)= R_k(\la,\om,v).
\ee
The induction claim, we have to prove, is  that the partial 
derivative $\partial_\la w_{k+1}(\la,\om,v)$ exists in $\CC_n$
and depends continuously on $\la$, $\om$, and $v$.
The candidate for  $\d_\la w_{k+1}(\la,\om,v)$ is
$$
\d_w\F(\la,\om,0,0)^{-1}\left(\d_\la R_k(\la,\om,v)-\d_w\d_\la \F(\la,\om,0,0)w_{k+1}(\la,\om,v)\right).
$$
This candidate  depends continuously on $\la$, $\om$, and $v$ because of \reff{strcont1} and  \reff{strcont2}.
In order to show that this candidate is really  $\d_\la w_{k+1}(\la,\om,v)$ we denote,
for the sake of shortness, $w_k(\la):=w_k(\la,\om,v)$,  $R_k(\la):=R_k(\la,\om,v)$, 
$\d_w\F(\la):=\d_w\F(\la,\om,0,0)$, $\d_\la\d_w \F(\la):=\d_\la\d_w\F(\la,\om,0,0)$ and 
calculate
\begin{eqnarray*}
\lefteqn{
\d_w\F(\la+\mu)\left(w_{k+1}(\la+\mu)- w_{k+1}(\la)
-\mu \d_w\F(\la)^{-1}\left(\d_\la R_k(\la)
-\d_\la\d_w \F(\la)w_{k+1}(\la)\right)\right)}
\nonumber\\
&&= R_k(\la+\mu)- R_k(\la)-\mu \d_w\F(\la+\mu) \d_w\F(\la)^{-1}\d_\la R_k(\la) \nonumber \\
&&-\left(\d_w\F(\la+\mu)-\d_w\F(\la)-
\mu \d_w\F(\la+\mu) \d_w\F(\la)^{-1}
\d_\la\d_w  \F(\la)\right)w_{k+1}.
\label{diff}
\end{eqnarray*}
The right-hand side is $o(\mu)$ in $\CC_n$ for $\mu \to 0$ because of \reff{strcont}. Hence, the  same is true for the 
left-hand side. Now, Lemma \ref{Magnus} yields 
the claim.
\end{subproof}

\begin{cclaim}\label{cl:16}
The sequence $\d_{\la}w_k(\la,\om,v)$ converges  in $\CC_n$
for $k \to \infty$  uniformly with respect to  $\la,\om$, and $v$.
\end{cclaim}

\begin{subproof}
From  
\reff{Feq2} it follows
\begin{eqnarray*}
\lefteqn{
\d_w\F(\la,\om,0,0)\left(\d_\la w_{k+1}(\la,\om,v)-\d_\la w_{k}(\la,\om,v)\right))}\\
&&+\d_w\d_\la \F(\la,\om,0,0)\left(w_{k+1}(\la,\om,v)-w_{k}(\la,\om,v)\right)\\
&&=-\d_w\F(\la,\om,v,w_k(\la,\om,v))\d_\la w_{k}(\la,\om,v)-\d_\la \F(\la,\om,v,w_k(\la,\om,v)).
\end{eqnarray*}
Therefore, $\d_\la w_{k+1}(\la,\om,v)-\left(I-\d_w\F(\la,\om,0,0)^{-1}\d_w\F(\la,\om,v,w_k(\la,\om,v))\right)\d_\la w_{k}(\la,\om,v)$
converges in $\CC_n$ for $k \to \infty$ uniformly with respect to $\la,\om$, and $v$
(here we use \reff{dla} and the claim of step two). Hence, Lemma \ref{A2} yields the claim.
\end{subproof}

A similar result is true with respect to the $\om$-derivative.
\begin{cclaim}\label{cl:17}
The partial derivatives $\d_\om w_k(\la,\om,v)$ exist in $\CC_n$ for all
$k=1,2,\dots$ and the sequence $\d_\om w_k(\la,\om,v)$ converges in $\CC_n$ 
for $k \to \infty$
uniformly with respect to $\la,\om$, and $v$.
\end{cclaim}

Now we consider partial derivatives of the second order.
\begin{cclaim}\label{cl:18}
The  sequence $\d_t^2 w_k(\la,\om,v)$ converges in $\CC_n$ for $k \to \infty$
uniformly with respect to $\la,\om$, and $v$. 
\end{cclaim}

\begin{subproof}
 From \reff{tF} we get
\begin{eqnarray}
\lefteqn{\!\!\!\!\!\!\!\!\!
\d_t^2\F(\la,\om,v,w)=\d_w\F(\la,\om,v,w)\d_t^2w+\d_w^2\F(\la,\om,v,w)(\d_tw,\d_tw)}\nonumber\\
&&
\!\!\!\!\!\!\!\!\!\!\!\!\!\!\!
+2\d_v\d_w\F(\la,\om,v,w)(\d_tv,\d_tw)+\d_v^2\F(\la,\om,v,w)(\d_tv,\d_tv)
+\d_v\F(\la,\om,v,w)\d_t^2v
\label{*}
\end{eqnarray}
for all $w \in  \tilde{\CC}_n^2 \cap \ker Q$. Hence,  it follows from \reff{Feq2} that the sequence
$\d_t^2 w_{k+1}(\la,\om,v)-\left(I-\d_w\F(\la,\om,0,0)^{-1}\d_w\F(\la,\om,v,w_k(\la,\om,v))\right)\d_t^2 w_{k}(\la,\om,v)$
converges in $\CC_n$ for $k \to \infty$ uniformly with respect to $\la,\om$, and $v$, and Lemma \ref{A2} yields the claim.
\end{subproof}

Similarly we get the following statement.
\begin{cclaim}\label{cl:19}
The sequences $\d_t\d_{v_j} w_k(\la,\om,v)$ and
$\d_{v_i}\d_{v_j} w_k(\la,\om,v)$
converge in $\CC_n$ for $k \to \infty$
uniformly with respect to $\la,\om$, and $v$.
\end{cclaim}

\begin{cclaim}\label{cl:110}
The partial derivatives $\d_t\d_\la w_k(\la,\om,v)$  exist in $\CC_n$ for all $k=1,2,\dots$ 
and the sequence $\d_t\d_\la w_k(\la,\om,v)$  converges in 
$\CC_n$ for $k \to \infty$
uniformly with respect to $\la,\om$, and $v$. 
\end{cclaim}

\begin{subproof}
We have
\begin{eqnarray}
&&\d_w\F(\la,\om,0,0)\left(\d_tw_{k+1}(\la,\om,v)-\d_tw_{k}(\la,\om,v)\right)\nonumber\\
&&=-\d_v\F(\la,\om,v,w_{k}(\la,\om,v))\d_tv
-\d_w\F(\la,\om,v,w_{k}(\la,\om,v))\d_tw_{k}(\la,\om,v).
\label{tla}
\end{eqnarray}
Hence, we can proceed as in Claims \ref{cl:15} and \ref{cl:16}, 
replacing \reff{sequence} by \reff{tla}, to get the desired statement.
\end{subproof}
A similar argument works to prove the next claim.

\begin{cclaim}\label{cl:111}
The partial derivatives  $\d_t\d_\om w_k(\la,\om,v)$,
$\d_{v_j}\d_\la w_k(\la,\om,v)$, and $\d_{v_j}\d_\om w_k(\la,\om,v)$
exist in $\CC_n$ for all $k=1,2,\dots$ and their sequences converge in 
$\CC_n$ for $k \to \infty$
uniformly with respect to $\la,\om$, and $v$.
\end{cclaim}

\begin{cclaim}\label{cl:112}
For all $k=1,2,\dots$ the partial 
derivative $\partial_\la^2 w_k(\la,\om,v)$ exists, belongs to $\CC_n$
and depends continuously
(with respect to $\|\cdot\|_\infty$) on $\la$, $\om$, and $v$.
\end{cclaim}
\begin{subproof}
Here we use Claims \ref{cl:18} and \ref{cl:110} and  induction on $k$. 

Similarly to Claims \ref{cl:15}  one can show that the restriction of the map $\F$ to
$[-\delta_0,\delta_0] \times  \R  \times \mbox{im}\,Q \times (\tilde{\CC}_n^2 \cap \ker Q)$ 
is twice differentiable (with respect to the norm $\|\cdot\|_2$ in $\tilde{\CC}_n^2 \cap \ker Q$). We denote 
the second partial derivative with respect to $\la$ of this restriction
by $\d_\la^2 \F$, i.e., $\d_\la^2 \F$ is a continuous map from $[-\delta_0,\delta_0] \times \R \times \mbox{im}\,Q \times (\tilde{\CC}_n^2 \cap \ker Q)$ into $\CC_n$.
By induction assumption the second partial derivative with respect to $\la$ of $R_k$  
exists in $\CC_n$ (because of $w_{k}(\la,\om,v) \in \tilde{\CC}_n^2$).
It follows from \reff{Feq2} that
\beq
\label{Feq3}
\d_w\F(\la,\om,0,0)\d_\la w_{k+1}(\la,\om,v)=\d_\la R_k(\la,\om,v)-\d_w\d_\la \F(\la,\om,0,0)w_{k}(\la,\om,v).
\ee
Hence, if $\d_\la^2 w_{k+1}(\la,\om,v)$ exists in $\CC_n$, then  
it holds 
\begin{eqnarray}
\lefteqn{\!\!\!\!\!\!\!\!\!\!\!\!\!\!\!\!
\d_w\F(\la,\om,0,0) \d_\la^2 w_{k+1}(\la,\om,v)=\d_\la^2 R_k(\la,\om,v)-\d_w\d_\la^2\F(\la,\om,0,0)w_{k}(\la,\om,v)}\nonumber\\
&&\!\!\!\!\!\!\!\!\!\!\!\!\!\!\!\!\!\!\!\!\!\!\!\!\!\!\!
-\d_w\d_\la \F(\la,\om,0,0)\d_\la w_{k}(\la,\om,v)-\d_w\d_\la \F(\la,\om,0,0)\d_\la w_{k+1}(\la,\om,v).
\label{**}
\end{eqnarray}
Here $\d_w\d_\la^2\F(\la,\om,0,0) \in \LL(\tilde{\CC}_n^2 \cap \ker Q;\CC_n)$ is the derivative of the map $\la \in [-\delta,\delta_0]
\mapsto \d_\la\d_w\F(\la,\om,0,0) \in \LL(\tilde{\CC}_n^2 \cap \ker Q;\CC_n)$ or, the same, the derivative of the map $w \in \tilde{\CC}_n^2 \cap \ker Q
\mapsto \d_\la^2 \F(\la,\om,0,w) \in \CC_n$ at $w=0$ (cf. \reff{dla}).

We use the notation $w_k(\la)$,  $R_k(\la)$, 
$\d_w\F(\la)$, $\d_\la\d_w \F(\la)$, introduced above, and  $\d_\la^2\d_w\F(\la):=\d_w\d_\la^2\F(\la,\om,0,0)$.
Because of \reff{**} the candidate for $\d_\la^2w_{k+1}(\la)$ is
$$
\d_w\F(\la)^{-1}(\d_\la^2 R_k(\la)
-\d_w\d_\la\F(\la)\d_\la w_{k}(\la)-\d_w\d_\la\F(\la)\d_\la w_{k+1}(\la)-\d_w\d_\la^2\F(\la)w_{k}(\la)).
$$
In order to show that this candidate is really  $\d^2_\la w_{k+1}(\la)$, we use \reff{Feq3} and calculate
\begin{eqnarray*}
\lefteqn{
\d_w\F(\la+\mu)\Big[ \d_\la w_{k+1}(\la+\mu)-  \d_\la w_{k+1}(\la)}\nonumber\\
&&-\mu \d_w\F(\la)^{-1}\left(\d_\la^2 R_k(\la)-\d_w\d_\la\F(\la)\d_\la w_{k}(\la)-\d_w\d_\la\F(\la)\d_\la w_{k+1}(\la)-
\d_w\d_\la^2\F(\la)w_{k}(\la)\right)\Big]
\nonumber\\
&&= \Bigl[\d_\la R_k(\la+\mu)-  \d_\la R_k(\la)-\mu \d_w\F(\la+\mu) \d_w\F(\la)^{-1}\d_\la^2 R_k(\la)\Bigr] \nonumber \\
&&-\Big[\d_w\d_\la\F(\la+\mu)-\d_w\d_\la\F(\la)- \mu \d_w\F(\la+\mu) \d_w\F(\la)^{-1}\d_w\d_\la^2\F(\la)\Big]
w_{k}(\la)\nonumber\\
&&-\Big[\d_w\d_\la \F(\la+\mu)\left(w_{k}(\la+\mu)-w_{k}(\la)\right)-
\mu \d_w\F(\la+\mu) \d_w\F(\la)^{-1}\d_w\d_\la \F(\la)
\d_\la w_{k}(\la)\Big]\nonumber\\
&&-\Big[ \d_w\F(\la+\mu)- \d_w\F(\la)- \mu \d_w\F(\la+\mu) \d_w\F(\la)^{-1} \d_\la \d_w\F(\la+\mu)\Big]
w_{k+1}(\la).\nonumber
\end{eqnarray*}
The right-hand side is $o(\mu)$ in $\CC_n$ for $\mu \to 0$ (here we use \reff{lawest}). Hence, 
the same is true for the left-hand side, and Lemma \ref{Magnus} yields 
the claim.
\end{subproof}

\begin{cclaim}\label{cl:113}
The sequence $\d_{\la}^2w_k(\la,\om,v)$ converges  in $\CC_n$
for $k \to \infty$  uniformly with respect to  $\la,\om$, and $v$. 
\end{cclaim}

\begin{subproof}
It follows from  
\reff{Feq2} that
\begin{eqnarray*}
\lefteqn{
\d_w\F(\la,\om,0,0)\left(\d_\la^2 w_{k+1}(\la,\om,v)-\d_\la^2 w_{k}(\la,\om,v)\right)}\\
&&+2\d_w\d_\la \F(\la,\om,0,0)\left(\d_\la w_{k+1}(\la,\om,v)-\d_\la w_{k}(\la,\om,v)\right)\\
&&+\d_w\d_\la^2 \F(\la,\om,0,0)\left(w_{k+1}(\la,\om,v)-w_{k}(\la,\om,v)\right)\\
&&=-\d_w\F(\la,\om,v,w_k(\la,\om,v))\d_\la^2 w_{k}(\la,\om,v)-2\d_w\d_\la \F(\la,\om,v,w_k(\la,\om,v)) \d_\la w_{k}(\la,\om,v)\\.
&&-\d_\la^2 \F(\la,\om,v,w_k(\la,\om,v)).
\end{eqnarray*}
Therefore, $\d_\la^2 w_{k+1}(\la,\om,v)-\left(I-\d_w\F(\la,\om,0,0)^{-1}\d_w\F(\la,\om,v,w_k(\la,\om,v))\right)\d_\la^2 w_{k}(\la,\om,v)$
converges in $\CC_n$ for $k \to \infty$ uniformly with respect to $\la,\om$, and $v$. Hence, Lemma \ref{A2} yields the claim.
\end{subproof}

Similarly one shows that the remaining second order partial derivatives $\d_\om^2w_k(\la,\om,v)$ 
and $\d_\om \d_\la w_k(\la,\om,v)$
exist in $\CC_n$ for all $k=1,2,\dots$ and  that their sequences converge in $\CC_n$ for $k \to \infty$
uniformly with respect to $\la,\om$, and $v$.
\end{proof}

\begin{rem}
\label{Cinfty}
In fact the map $\hat{w}$ is not only $C^2$-smooth, as it is claimed in Lemma \ref{laregular}, but  $C^\infty$-smooth.
In order to prove this rigorously,  one has to handle ``higher order analogues'' of  formulas like  
\reff{*} and \reff{**}, which are getting more and more complicated.
\end{rem}

\section{The bifurcation equation}
\label{bifeq}
\renewcommand{\theequation}{{\thesection}.\arabic{equation}}
\setcounter{equation}{0} 
In this section we finish the proof of Theorem \ref{thm:hopf} by inserting the solution $w=\hat{w}(\la,\om,v)$
of the infinite dimensional part \reff{infinite} of the Liapunov-Schmidt system into the  finite dimensional part \reff{finite}
\beq
\label{bif}
P((I-C(\la,\om))(v+\hat{w}(\la,\om,v))-F(\la,\om,v+\hat{w}(\la,\om,v)))=0
\ee
and by solving this so-called bifurcation equation with respect to $\la \approx 0$ and $\om \approx 1$
for given small $v \not=0$.

\subsection{Local solution of \reff{bif}}
Because of Lemma \ref{kerim} the functions 
${\bf v_1}=\Re {\bf v}$ and   ${\bf v_2}=\Im {\bf v}$ constitute  a basis in $\ker L_0$. Hence, the variable $v \in \ker L_0$ can be represented as
$$
v=\Re(\zeta{\bf v}) \mbox{ with } \zeta \in \C.
$$
Moreover, because of ${\bf w_1}=\Re {\bf w}$ and   ${\bf w_2}=\Im {\bf w}$
and of the definition \reff{Pdef} of the projection $P$ we have $Pu=0$ for $u \in \CC_n$ if and only if $\langle u,{\bf \tilde{A}w}\rangle =0$.
Hence, \reff{bif} is equivalent to the
one-dimensional complex equation
\beq\label{3.4}
\Phi(\la,\om,\zeta) = 0
\ee
with
\begin{eqnarray*}
&&\Phi(\la,\om,\zeta):=\\ 
&&
\left\langle
\left(I-C(\la,\om)\right)\left(\Re(\zeta{\bf v})+\hat w(\la,\om,\Re(\zeta{\bf v}))\right)
-F\left(\la,\om,\Re(\zeta{\bf v})+\hat w(\la,\om,\Re(\zeta{\bf v}))\right), \tilde{A}{\bf w}
\right\rangle.
\end{eqnarray*}
Here and in what follows $\langle\cdot,\cdot\rangle: \C^n\times\C^n \to \C$ is the Hermitian scalar product in $\C^n$.
Remark that $\Phi(\la,\om,\cdot):\C \to \C$ is $C^\infty$-smooth in the sense of real differentiability, i.e., as a map from a two-dimensional real vector space into itself.
It follows from \reff{42} that
\beq
\label{sym}
e^{i\vphi}\Phi(\la,\om,\zeta) = \Phi\left(\la,\om,e^{i\vphi}\zeta\right).
\ee
Therefore, it suffices to solve (\ref{3.4})  for real $\zeta$. 
To shorten notation, we will write
$$
\bar w(\la,\om,r):=\hat w(\om,\la,r\Re{\bf v}) \mbox{ for } r \in \R.
$$
Moreover, because of $\hat w(\la,\om,0)=0$ we have  $\Phi(\la,\om,0)=0$ for all  $\la$ and  $\om$.
Hence, it suffices to solve the so-called reduced or scaled bifurcation equation
\beq\label{redbif}
\Psi(\la,\om,r)=0
\ee
for small $r \in \R$. Here the map $\Psi$ is defined by
$$
\Psi(\la,\om,r) := \left\{
\begin{array}{lcr}
\displaystyle \frac{1}{r}\Phi(\la,\om,r) &\mbox{for}& r\not=0,\\
\displaystyle \lim_{s \to 0}\frac{1}{s}\Phi(\la,\om,s) &\mbox{for}& r=0,
\end{array}
\right.
$$
which leads to
\begin{eqnarray*}
&&\Psi(\la,\om,r) = \int_0^1\d_\zeta\Phi(\la,\om,rs)\,ds\\
&&=\int_0^1\Bigl\langle \left(I-C(\la,\om)-\d_uF\left(\la,\om,rs\Re{\bf v}+\bar w(\la,\om,rs)\right)\right)
\left(\Re{\bf v}+\d_r\bar w(\la,\om,rs)\right),\tilde Aw\Bigr\rangle
\,ds.
\end{eqnarray*}
Because of Lemma \ref{infIFT}(ii) we have 
\beq
\label{dwzero}
\d_r\bar w(\la,\om,0)=0.
\ee
This together with Lemma \ref{kerim} yields
$\Psi(0,1,0)=0.$
In order to solve \reff{redbif} with respect to $\la$ and $\om$ close to the solution $\la=0$, $\om=1$, $r=0$ by means 
of the implicit function theorem we have to calculate the partial derivatives $\d_\la\Psi(0,1,0)$ and  $\d_\om\Psi(0,1,0)$.

For $\la \approx 0$ denote by $\hat{\nu}(\la)$ the eigenvalue close to $\nu=-i$ of the eigenvalue problem \reff{adevp}--\reff{adevpbc}
(which is the complex conjugate to the eigenvalue  close to $\mu=i$ of the eigenvalue problem \reff{evp}),
and let $\hat{w}(\la): [0,1]\to \C^n$ be the corresponding eigenfunction, normalized by (cf. \reff{normed})
$$
\sum\limits_{j=1}^n\int_0^1v_j^0(x)\overline{\hat{w}_j(\la)(x)}\,dx=2.
$$
Further, define ${\bf \hat{w}(\la)}: [0,1] \times \R \to \C^n$ by  ${\bf \hat{w}(\la)}(x,t):= \hat{w}(\la)(x) e^{-it}$.
Then we have $\hat{w}(0)=w^0$ and, hence,  ${\bf \hat{w}(0)}={\bf w}$. Because $\hat{w}(\la)$ satisfies the adjoint boundary conditions 
\reff{adevpbc}, the function ${\bf \hat{w}(\la)}$ satisfies the adjoint boundary conditions
\reff{adbc}. Moreover, $\left(I-C(\la,1)-\d_uF\left(\la,1,0\right)\right)\Re{\bf v}$
satisfies the boundary conditions \reff{eq:1.2}. Hence, we get from
Lemma \ref{AB} that
\begin{eqnarray*}
&&\left\langle \left(I-C(\la,1)-\d_uF\left(\la,1,0\right)\right)\Re{\bf v},\tilde A(\la,1){\bf \hat{w}(\la)}\right\rangle\\
&&= \left\langle  A(\la,1)\left(I-C(\la,1)-\d_uF\left(\la,1,0\right)\right)\Re{\bf v},{\bf \hat{w}(\la)}\right\rangle\\
&&= \left\langle \left(A(\la,1)+B(\la)\right)\Re{\bf v},{\bf \hat{w}(\la)}\right\rangle
= \left\langle \Re{\bf v}, \left(\tilde{A}(\la,1)+\tilde{B}(\la)\right){\bf \hat{w}(\la)}\right\rangle\\
&&= \left\langle \Re{\bf v}, \left(i+\nu(\la)\right){\bf \hat{w}(\la)}\right\rangle
=\overline{\nu(\la)}-i.
\end{eqnarray*}
But $\left(A(0,1)+B(0)\right)\Re{\bf v}=0$ (cf. Lemma  \reff{kerim}), hence
\begin{eqnarray*}
\d_\la\Psi(0,1,0)
& =& \frac{d}{d\la}\left\langle \left(I-C(\la,1)-\d_uF\left(\la,1,0\right)\right)\Re{\bf v},\tilde A{\bf w}\right\rangle\Big|_{\la=0}\\
& =& \frac{d}{d\la}\left\langle \left(I-C(\la,1)-\d_uF\left(\la,1,0\right)\right)\Re{\bf v},\tilde A(\la,1){\bf \hat{w}(\la)}\right\rangle\Big|_{\la=0}\\
& =& \left\langle \left(\d_\la A(0,1)+\d_\la B(0)\right)\Re{\bf v},{\bf w}\right\rangle\\
& =& \frac{1}{2}\sum_{j=1}^n\int_0^1\left(\d_\la a_j(x,0)v^0_j(x)+\sum_{k=1}^n\d_{u_k}\d_\la b_j(x,0,0)v^0_k(x)\right)\overline{w^0_j(x)} dx.
\end{eqnarray*}
In particular, it holds 
\beq
\label{trans}
\Re \d_\la\Psi(0,1,0)=\Re \nu'(0)=\alpha
\ee
(cf. \reff{eq:IV}).
Similarly we have 
\begin{eqnarray}
\d_\om\Psi(0,1,0)
& =& \frac{d}{d\om}\left\langle \left[I-C(0,\om)-\d_uF\left(0,\om,0\right)\right]\Re{\bf v},\tilde A{\bf w}\right\rangle\Big|_{\om=1}\nonumber\\
& =& \left\langle \d_\om A(0,1)\Re{\bf v},{\bf w}\right\rangle  = -i\left\langle  \Re{\bf v},{\bf w}\right\rangle  = -i.
\label{im}
\end{eqnarray}
Hence, the transversality condition \reff{eq:IV} yields
$$
\det\frac{\partial(\Re \Psi, \Im \Psi)}{\partial(\la,\om)}(0,1,0)= -\alpha\not=0,
$$
therefore the implicit function theorem works. Using, moreover, that \reff{sym} yields $\Phi(\la,\om,-\zeta)=-\Phi(\la,\om,\zeta)$ and, hence,  $\Psi(\la,\om,-r)=\Psi(\la,\om,r)$, we get
\begin{lemma}\label{lem:bif}
There exist $\eps_0>0$ and $\delta>0$ such that for all $r \in [0,\eps_0]$ there exists exactly one solution $\la=\hat{\la}(r) \in  [-\delta,\delta]$,  $\om=\hat{\om}(r)
\in  [1-\delta,1+\delta]$
to \reff{redbif}. Moreover, the map  $r \in [0,\eps_0] \mapsto (\hat{\la}(r),\hat{\om}(r)) \in \R^2$ is $C^2$-smooth, and it holds $\hat{\la}(0)=0$, $\hat{\om}(0)=1$, and 
\beq
\label{firstDer}
\hat{\la}'(0)=\hat{\om}'(0)=0.
\ee
\end{lemma}
Now the local existence and uniqueness assertions of Theorem \ref{thm:hopf} are proved with
$$
[\hat{u}(\eps)](x,t)=\left[\eps\Re{\bf v}+\bar{w}\left(\hat{\la}(\eps),\hat{\om}(\eps),\eps\right)\right](x,t)=\eps\Re\left(e^{-it}v^0(x)\right)+O(\eps^2).
$$

\subsection{Bifurcation formula}

In this subsection we will prove the so-called bifurcation formula (cf. notation \reff{eq:IV} and \reff{coef})
\beq
\label{secondDer}
\hat{\la}''(0)=\frac{\beta}{\alpha}.
\ee
This formula determines the so-called bifurcation direction, i.e., if $\alpha \beta>0$, then the bifurcating periodic solutions to
(\ref{eq:1.1})--(\ref{eq:1.3}) exist for small $\la >0$, and,  if $\alpha \beta<0$, then they exist for small $\la <0$.
Moreover, it turns out that, as in the case of Hopf bifurcation for ODEs and parabolic PDEs, this formula determines the stability of the 
 bifurcating periodic solutions to
(\ref{eq:1.1})--(\ref{eq:1.3}): If  $\beta>0$ (the so-called supercritical case) and if the real parts of all eigenvalues $\mu \not=\pm i$ of \reff{evp} with $\la=0$
have negative real parts (this is a reinforcement of the nonresonance assumption \reff{nonres}),
then  the bifurcating periodic solutions to
(\ref{eq:1.1})--(\ref{eq:1.3}) are asymptotically orbitally stable,  
if $\beta<0$ (the so-called subcritical case) then they are unstable.
Remark that in the present paper we do not deal with the stability question because we do not have suitable criteria for nonlinear stability of time-periodic solutions to semilinear 
dissipative hyperbolic PDEs. This will be the topic for future work.

In the following calculations we will use a standard approach (cf., e.g., \cite[Chapter I.9]{Ki}) as well as its concrete realization for hyperbolic systems (see \cite{Gajewski}).

In order to prove \reff{secondDer}, we differentiate the identity $\Psi(\hat{\la}(r),\hat{\om}(r),r)=0$ twice with respect to $r$ at $r=0$.
Using \reff{firstDer}, we get 
$
\d_\la\Psi(0,1,0)\hat{\la}''(0)+\d_\om\Psi(0,1,0)\hat{\om}''(0)+\d_r^2\Psi(0,1,0)=0.
$
Hence \reff{trans} and \reff{im} yield
$
\alpha\hat{\la}''(0)+\mbox{Re}\,\d_r^2\Psi(0,1,0)=0.
$
On the other hand, 
differentiating the identity $r\Psi(\la,\om,r)=\Phi(\la,\om,r)$ three times with respect to $r$ at $\la=r=0$ and $\om=1$, we get
$
3\d_r^2\Psi(0,1,0)=\d_r^3\Phi(0,1,0).
$
Hence, in order to show \reff{secondDer} we have to show
\beq
\label{biform1}
\beta=- \frac{1}{3}\mbox{Re}\,\d_r^3\Phi(0,1,0).
\ee
To get a formula for $\d_r^3\Phi(0,1,0)$ we denote $\tilde{w}(r):=\bar w(0,1,r)$ and
calculate
\begin{eqnarray}
&&\Phi(0,1,r)=
\langle
\left(I-C(0,1)\right)\left(r\Re{\bf v}+\tilde{w}(r)\right)
-F\left(0,1,r\Re{\bf v}+\tilde{w}(r)\right), \tilde{A}{\bf w}
\rangle\nonumber\\
&&=
\left\langle
A\left(\left(I-C(0,1)\right)\left(r\Re{\bf v}+ \tilde{w}(r)\right)
-F\left(0,1,r\Re{\bf v}+\tilde{w}(r)\right)\right), {\bf w}
\right\rangle\nonumber\\
&&=
\left\langle
A\left(r\Re{\bf v}+\tilde{w}(r)\right)
-AF(0,1,r\Re{\bf v}+\tilde{w}(r)), {\bf w}
\right\rangle\nonumber\\
&&=
-\left\langle
B\left(r\Re{\bf v}+\tilde{w}(r)\right)
+AF(0,1,r\Re{\bf v}+\tilde{w}(r)), {\bf w}
\right\rangle.
\label{tildew}
\end{eqnarray}
Here we used that for all $u \in \CC_n$ the function $C(0,1)u
+F(0,1,u)$ satisfies the boundary conditions  \reff{eq:1.2} and 
that also ${\bf v}$ and $\tilde{w}(r)$ satisfy those  boundary conditions.
For  $\tilde{w}(r)$ this  follows from 
\begin{eqnarray}
&&P\left[(I-C(0,1))(r\Re{\bf v}+\tilde{w}(r))-F(0,1,r\Re{\bf v}+\tilde{w}(r))\right]\nonumber\\
&&=\sum_{k=1}^2 \langle((I-C(0,1))(r\Re{\bf v}+\tilde{w}(r))-F(0,1,r\Re{\bf v}+\tilde{w}(r))),\tilde{A} {\bf w_k} \rangle {\bf \tilde{v}_k}
\label{wbareq}
\end{eqnarray}
(cf. \reff{Pdef} and \reff{infinite}) and from the fact, that the functions $ {\bf \tilde{v}_k}$ satisfy the  boundary conditions  \reff{eq:1.2}
(cf. \reff{vtilde}.
Now we differentiate three times \reff{tildew} with respect to $r$ at $r=0$, use $B+A\d_u F(0)=0$ and get
\begin{eqnarray}
&&\d_r^3\Phi(0,1,0)\nonumber\\
&&=-\left\langle A\d_u^3F(0,1,0)(\Re{\bf v},\Re{\bf v},\Re{\bf v})+3A\d_u^2F(0,1,0)(\Re{\bf v},\tilde{w}''(0)), {\bf w}
\right\rangle.
\label{d3Phi}
\end{eqnarray}

Let us calculate $\tilde{w}''(0)$. For that we differentiate the identity \reff{wbareq} twice with respect to $r$ at $r=0$
and get
\begin{eqnarray}
&&P\left[(I-C(0,1)-\d_uF(0,1,0))\tilde{w}''(0)- \d_u^2F(0,1,0)( \Re{\bf v}, \Re{\bf v})\right]\nonumber\\
&&=-\sum_{k=1}^2\langle A\d_u^2F(0,1,0)( \Re{\bf v}, \Re{\bf v}),{\bf w_k}\rangle {\bf \tilde{v}_k}.
\label{d2weq}
\end{eqnarray}
But $[A \d_u^2F(0,1,0)( \Re{\bf v}, \Re{\bf v})](x,t)$ is of the type $c_1(x)e^{2it}+c_0(x)+c_{-1}(x)e^{-2it}$
(cf. \reff{Af}), and ${\bf w_k}(x,t)$ is of the type  $d_1(x)e^{it}+d_{-1}(x)e^{-it}$, 
hence $\langle A \d_u^2F(0,1,0)(\Re{\bf v}, \Re{\bf v}),{\bf w_k}\rangle=0$.
Applying $A$ to both sides of \reff{d2weq} and using Lemma \ref{AB},
we get
\beq
\label{ABeq1}
(A+B)\tilde{w}''(0)= A \d_u^2F(0,1,0)(\Re{\bf v}, \Re{\bf v}).
\ee
Because of \reff{Af} the $j$-th component of the right-hand side of \reff{ABeq1} calculated at the point $(x,t)$ is
\begin{eqnarray*}
&&-\sum_{k,l=1}^n b_{jkl}(x) \Re(v_k^0(x)e^{-it}) \Re(v_l^0(x)e^{-it})\\
&&=
-\frac{1}{4}\sum_{k,l=1}^n b_{jkl}(x)\left(v_k^0(x)v_l^0(x)e^{-2it}+2v_k^0(x)\overline{v_l^0(x)}+\overline{v_k^0(x)}\overline{v_l^0(x)}e^{2it}\right).
\end{eqnarray*}
Here we used notation \reff{bb}.
The  $j$-th component of the left-hand side of \reff{ABeq1} calculated at the point $(x,t)$ is
$$
\d_t\tilde{w}_j''(0)(x,t)+a_j(x,0)\d_x\tilde{w}_j''(0)(x,t)+\sum_{k=1}^nb_{jk}(x)\tilde{w}_k''(0)(x,t).
$$
We are going to solve \reff{ABeq1} with respect to $\tilde{w}''(0)$  by means of the ansatz 
$$
\tilde{w}_j''(0)(x,t)=y_j(x)e^{-2it}+z_j(x)+\overline{y_j(x)}e^{2it}
$$
with $C^1$-functions $y_j:[0,1] \to \C$ and  $z_j:[0,1] \to \R$ which satisfy the boundary conditions corresponding to \reff{eq:1.2}. 
Then  \reff{ABeq1} is equivalent to
\begin{eqnarray*}
a_j(x,0) \frac{d}{dx}y_j(x) - 2iy_j(x) + \sum_{k=1}^nb_{jk}(x)y_k(x) &=& -\frac{1}{4}\sum_{k,l=1}^nb_{jkl}(x)v_k^0(x)v_l^0(x),  \nonumber\\
 y_j(0)&=&\sum_{k=m+1}^n r_{jk} y_k(0), \ j=1,\dots,m \\
 y_j(1)&=&\sum_{k=1}^mr_{jk}y_k(1), j=m+1,\dots ,n, \nonumber
\end{eqnarray*}
\begin{eqnarray*}
a_j(x,0) \frac{d}{dx}z_j(x) + \sum_{k=1}^nb_{jk}(x)z_k(x) &=& -\frac{1}{2}\sum_{k,l=1}^nb_{jkl}(x)v_k^0(x)\overline{v_l^0(x)},\nonumber \\
 z_j(0)&=&\sum_{k=m+1}^n r_{jk} z_k(0),  j=1,\dots,m \\
 z_j(1)&=&\sum_{k=1}^mr_{jk}z_k(1), j=m+1,\dots ,n.
\end{eqnarray*}
In other words: The functions $y_j$ and $z_j$ have to be those as introduced in Section \ref{sec:results}.

Now we calculate the two terms in \reff{d3Phi}: The first term is
$$
-\left\langle A\d_u^2F(0,1,0)(\Re{\bf v},\Re{\bf v},\Re{\bf v}),{\bf w}\right\rangle=
\frac{3}{8}\sum_{j,k,l,r=1}^n\int_0^1 b_{jklr}v_k^0v_l^0\overline{v_r^0}\overline{w_j^0}dx.
$$
The second term is
$$
-3\left\langle A\d_u^2F(0,1,0)(\Re{\bf v},\tilde{w}''(0)),{\bf w}\right\rangle
=\frac{3}{2}\sum_{j,k,l=1}^n\int_0^1 b_{jkl}\left(v_k^0z_l+\overline{v_k^0}y_l\right)\overline{w_j^0}dx.
$$
Therefore, \reff{biform1} and \reff{d3Phi} yield the formula \reff{coef} of Section \ref{sec:results}.

\section{Example}
\label{Examples}
\renewcommand{\theequation}{{\thesection}.\arabic{equation}}
\setcounter{equation}{0}
In this section we present a simple and, hence, academic example of a problem of the type  
(\ref{eq:1.1})--(\ref{eq:1.3}). The aim is to show that the dissipativity condition
 \reff{Fred} does not contradict to the usual Hopf bifurcation assumptions
 \reff{smooth}--\reff{hyp}, \reff{u=0}, and  \reff{geo}--\reff{nonres} of Theorem \ref{thm:hopf}. For examples modelling processes in natural sciences  and technology see Section~1.2.
The equations depend, besides of the frequency parameter $\om$ and the bifurcation parameter $\la$, on an additional real parameter $\gamma$ which can be chosen in such a 
way that supercritical as well as subcritical Hopf bifurcation  occurs:
\beq
\label{example}
\begin{array}{l}
\om \d_tu_1-\d_xu_1+\la u_1-u_2+\ga u_1^3=\om \d_tu_2+\d_xu_2=0,\;x \in (0,1),\\
u_1(0,t)=0 , \quad u_1(1,t)=u_2(1,t).
\end{array}
\ee
Using the notation of Section \ref{sec:intr}, we have $m=1$, $n=2$, $a_1(x,\la)=-1$, $a_2(x,\la)=1$,  $b_1(x,\la,u_1,u_2)=\la u_1-u_2+\ga u_1^3$, $b_2(x,\la,u_1,u_2)=0$,
$r_{12}=0$, and $r_{21}=1$. 

Obviously, the conditions \reff{smooth}--\reff{hyp} and  \reff{u=0} are fulfilled. Condition  \reff{Fred} is fulfilled also because the choice $r_{12}=0$ implies $R_0=0$ (cf. \reff{Rdef0}).

In order to check the remaining conditions  \reff{geo}--\reff{nonres} we consider the eigenvalue problem 
(see \reff{evp})
\beq
\label{exampleevp}
\begin{array}{l}
-v_1'+(\la-\mu) v_1-v_2=v_2'-\mu v_2=0,\;x \in (0,1),\\
v_1(0)=v_1(1)-v_2(1)=0
\end{array}
\ee
and  the adjoint eigenvalue problem 
(see \reff{adevp})
\beq
\label{exampleadevp}
\begin{array}{l}
w_1'+(\la-\nu) w_1=-w_2'-w_1-\nu w_2=0,\;x \in (0,1),\\
w_2(0)=w_1(1)-w_2(1)=0.
\end{array}
\ee
It is easy to verify that there do not exist real eigenvalues $\mu$ to \reff{exampleevp} and that \reff{exampleevp} is equivalent to
\beq
\label{exampleevp1}
\begin{array}{l}
\displaystyle v_1(x)=\frac{c}{\la-2\mu}\left(e^{\mu x} -e^{(\la-\mu)x}\right),\; v_2(x)=ce^{\mu x}, \\
e^{\la-2\mu}=2\mu-\la+1.
\end{array}
\ee
Here $c=v_2(0)$ is a nonzero complex constant.
Setting $\la-2\mu=a+ib$ with $a \in \R$ and (without loss of generality) $b>0$, we get 
$$
b=\sqrt{e^{2a}-(1-a)^2}
$$ 
and
\beq
\label{chareq}
\sin \sqrt{e^{2a}-(1-a)^2}=- \sqrt{1-e^{-2a}(1-a)^2}.
\ee
It is easy to see that equation \reff{chareq} has (besides of the solution $a=0$) a countable number of solutions $0<a_0<a_1\ldots$ tending to $\infty$.
Hence, the spectrum  
of \reff{exampleevp} consists of countably many geometrically simple eigenvalues
$$
\mu_j^\pm(\la)=\frac{1}{2}\left(\la-a_j\pm i\sqrt{e^{2a_j}-(1-a_j)^2}\right).
$$
If $\lambda=a_0$, then the eigenvalue pair $\mu^\pm_0(a_0)$ is on the imaginary axis:
$$
\pm i\om_0:=\mu_0^\pm(a_0)=\pm\frac{i}{2}\sqrt{e^{2a_0}-(1-a_0)^2}.
$$
The real parts of all other eigenvalues are negative,
hence the nonresonance condition \reff{nonres} is fulfilled.
Obviously, the transversality condition \reff{eq:IV} is fulfilled, since
$$
\al=\frac{d}{d\la}\mbox{Re}\mu^\pm_0(\la)|_{\la=a_0}=\frac{1}{2}.
$$

In order to check \reff{alg},  we first calculate $w_1(x)$ and 
 $w_2(x)$ using \reff{exampleadevp}
and the fact that $\nu=\bar \mu$. 
We get 
$$
w_1(x)=de^{(-i\om_0-a_0)x},\; 
w_2(x)=\frac{d}{a_0+2i\om_0}\left(e^{-(i\om_0+a_0) x} - e^{i\om_0 x}\right).
$$
Here $d=w_1(0)$ is again a nonzero complex constant.
Hence,
\begin{eqnarray}
\int_0^1\left(v_1\overline{w}_1+v_2\overline{w}_2\right)\, dx&=&
\frac{2c\bar{d}}{a_0-2i\om_0}\left(\int_0^1e^{(2i\om_0-a_0)x}dx-1\right)\nonumber\\
&=&
\frac{2c\bar{d}}{(a_0-2i\om_0)^2}\left(e^{a_0-2i\om_0}-e^{2i\om_0-a_0}\right)\not=0.
\label{Null0}
\end{eqnarray}
Here we used that 
\beq
\label{Null}
e^{a_0-2i\om_0}=2i\om_0-a_0+1
\ee
(cf. \reff{exampleevp1}).

In order to calculate $\beta$ we have to normalize $v_1,v_2,w_1$ and $w_2$ according to \reff{normed}. Using \reff{Null0}, we get
\beq
\label{Null1}
c\bar{d}=\frac{(a_0-2i\om_0)^2}{e^{a_0-2i\om_0}-e^{2i\om_0-a_0}}.
\ee
Further, we have $b_{jkl}=0$ for all  indices $j,k$, and $l$, $b_{1111}=6\gamma$ and $b_{jklr}=0$ if one of the 
indices $j,k,l,r$ is not equal to one. 
Hence,
$$
\beta=-\frac{3\ga}{4}\Re\int_0^1v_1^0v_1^0\overline{v_1^0}\overline{w_1^0}\,dx.
$$
Moreover,  using \reff{Null} and \reff{Null1}, we get
$$
\left|v_1^0(x)\right|^2v_1^0(x)\overline{w_1^0(x)}= 
\frac{|c|^2(e^{a_0-2i\om_0}+e^{2a_0})}{|a_0-2i\om_0|^2|1+e^{a_0-2i\om_0}|^2}\left|1-e^{(a_0-2i\om_0)x}\right|^2\left(1-e^{(2i\om_0-a_0)x}\right).
$$
Equation  \reff{Null} yields also
$$
\int_0^1\left|1-e^{(a_0-2i\om_0)x}\right|^2\left(1-e^{(2i\om_0-a_0)x}\right)dx=
5+\frac{e^{2a_0}-1}{2a_0}+\frac{e^{4i\om_0}-1}{4i\om_0}+e^{2i\om_0-a_0}
$$
and
\begin{eqnarray*}
&&\mbox{Re}\left(\left(e^{a_0-2i\om_0}+e^{2a_0}\right)\left(5+\frac{e^{2a_0}-1}{2a_0}+\frac{e^{4i\om_0}-1}{4i\om_0}+e^{2i\om_0-a_0}\right)\right)\\
&&=\left(e^{2a_0}-a_0+1\right)\left(5+\frac{e^{2a_0}-1}{2a_0}\right).
\end{eqnarray*}
Therefore,
$$
\beta=-\frac{3\ga|c|^2(e^{a_0}-a_0+1)}{4|a_0-2i\om_0|^2|1+e^{a_0-2i\om_0}|^2}\left(5+\frac{e^{2a_0}-1}{2a_0}\right).
$$
Now, because of $a_0>0$ it follows $\beta >0$ (supercritical Hopf bifurcation) for $\ga<0$ and  $\beta <0$ (subcritical Hopf bifurcation) for $\ga>0$.

\appendix
\section{
}
\label{Appendix}
\renewcommand{\theequation}{{\thesection}.\arabic{equation}}
\setcounter{equation}{0}
In this appendix we present a simple linear version of the so-called fiber contraction principle
(see, e.g., \cite[Section 1.11.3]{Chicone}):

\begin{lemma}
\label{A2}
Let $U$ be a Banach space and $u_1,u_2,\ldots \in U$ a converging sequence. Further, let $A_1,A_2,\ldots \in \LL(U)$
be a sequence of linear bounded operators on $U$ such that there exists $c<1$ such that for all $u \in U$ it holds
\beq
\label{a1}
\|A_n u\| \le c \|u\| \mbox{ for all } n=1,2,\ldots
\ee
and \beq
\label{a2}
A_1u , A_2u,\ldots \,\,\mbox{\rm converges in  }  U.
\ee
Finally, let $v_1,v_2,\ldots \in U$ be a sequence such that for all $n=1,2,\ldots$ we have 
\beq
\label{a3}
v_{n+1}=A_nv_n+u_n.
\ee
Then the sequence  $v_1,v_2,\ldots$ converges in $U$.
\end{lemma}
\begin{proof}
Because of \reff{a2} there exists $A \in \LL(U)$ such that for all $u \in U$ we have $A_nu \to Au$ in $U$ for $n \to \infty$.
Moreover, \reff{a1} yields that $\|Au\| \le c\|u\|$ for all $u \in U$. Because of $c < 1$ there exists exactly one $v \in U$ such that
$v=Av+u$, where $u \in U$ is the limit of the sequence $u_1,u_2,\ldots \in U$.

Let us show that $v_k \to v$ in $U$  for $k \to \infty$.
Because of \reff{a3} we have 
\beq
\label{a}
v_{n+1}-v
=A_n(v_{n}-v)+\left(A_n-A\right)v+u_{n}-u.
\ee
Using the notation
$r_{n}:=\|v_n-v\|$ and 
$s_{n}:=\|\left(A_{n}-A\right)v\| + \|u_{n}-u\|$,
we get from \reff{a1} and \reff{a}
\beq
r_{n+1} \le c r_{n} +s_{n} \le c^2  r_{n-1}+cs_{n-1}+s_{n}
\le \ldots\le c^{n}r_{1}+\sum_{m=0}^n c^{n-m} s_{m}.
\label{riteration}
\ee
Take $s_0>0$ such 
that $s_{n} \le s_0$ for all $n$.
Then \reff{riteration} implies for $0 \le l \le n$
\beq
r_{n+1} 
\le c^n \left(r_1+s_0\sum_{m=0}^l c^{-m}\right)+\frac{1}{1-c}\max_{m>l}s_{m}.
\label{riteration1}
\ee
But we have
\beq
\label{max}
\max_{m>l}s_{m} \to 0 \mbox{ for } l \to \infty.
\ee
Given $\eps>0$, fix $l\in\N$ sufficiently large, in order to meet the 
estimate
\beq
\label{max1}
\frac{1}{1-c}\max_{m>l}s_{m} <\frac{\eps}{2}.
\ee
Moreover, choose $n_0\in\N$ so large that 
\beq
\label{sum}
 c^n \left(r_1+s_0\sum_{m=0}^l c^{-m}\right) <\frac{\eps}{2} \mbox{ for all } n\ge n_0.
\ee
Then \reff{riteration1}--\reff{sum} yield that $|r_{n+1}|<\eps$ for all $ n\ge n_0$, as desired.
\end{proof}

\section*{Acknowledgments}
The first author  was supported by the Alexander von Humboldt Foundation.  
Both authors acknowledge support of the DFG Research Center {\sc Matheon} mathematics for key technologies (project D8).

\end{document}